\documentclass[11pt]{article}

\usepackage{etex}
\usepackage{graphics}
\usepackage{color}
\usepackage{cancel}
\usepackage{stmaryrd}
\usepackage{amsmath}
\usepackage{amsthm}
\usepackage{amssymb}
\usepackage{verbatim,cmll}
\usepackage{setspace}
\usepackage{lscape}
\usepackage{latexsym,relsize}
\usepackage{amscd}
\usepackage{multicol}
\usepackage{bussproofs}
\usepackage[position=top]{subfig}
\usepackage{leqno}
\usepackage{tikz}
\usepackage{authblk}
\usepackage{multicol}
\usetikzlibrary{arrows}
\usetikzlibrary{matrix}
\usetikzlibrary{patterns}
\usetikzlibrary{shapes}
\usepackage{bbm}
 \usepackage[bb=boondox]{mathalfa}
\usepackage{enumerate}

\usepackage[left=3cm,top=3cm,right=3cm,bottom=2cm]{geometry}

\newcommand{\exts}[1]{[\![{#1}]\!]}
\newcommand{\ints}[1]{(\![{#1}]\!)}
\EnableBpAbbreviations

\renewcommand{\epsilon}{\varepsilon}

\newcommand{\bba}{\mathbb{A}}
\newcommand{\bbA}{\mathbb{A}}

\newcommand{\marginnote}[1]{\marginpar{\raggedright\tiny{#1}}}
\newcommand{\bari}[1]{\overline{#1}^{\, i}}

\tikzset{
	treenode/.style = {align=center, inner sep=0pt, text centered},
	Ske/.style = {treenode, ellipse, double, draw=black,
		minimum width=6pt, thick},
	PIA/.style = {treenode, ellipse, black, draw=black,
		minimum width=6pt},
	Crit/.style = {treenode, rectangle, draw=black,
		minimum width=0.5em, minimum height=0.5em}
}

\theoremstyle{plain}
\newtheorem{thm}{Theorem}

\newtheorem{lem}[thm]{Lemma}
\newtheorem{cor}[thm]{Corollary}
\newtheorem{prop}[thm]{Proposition}

\newtheorem{lemma}[thm]{Lemma}
\newtheorem{remark}[thm]{Remark}
\newtheorem{example}[thm]{Example}

\theoremstyle{definition}
\newtheorem{definition}[thm]{Definition}

\makeindex
\title{Goldblatt-Thomason for $\mathrm{LE}$-logics \thanks{This research is supported by the NWO Vidi grant 016.138.314, the NWO Aspasia grant 015.008.054, a Delft Technology Fellowship awarded to the second author in 2013. The authors would also like to
thank Robert Goldblatt for his very insightful and useful comments on a draft of this paper.}}

\author[1]{Willem Conradie}
\author[2,3]{Alessandra Palmigiano}
\author[2]{Apostolos Tzimoulis}
\affil[1]{University of the Witwatersrand, South Africa}
\affil[2]{Delft University of Technology, the Netherlands}
\affil[3]{University of Johannesburg, South Africa}

\date{}

\begin{document}
\maketitle

\begin{abstract}
We prove a uniform version of the Goldblatt-Thomason theorem for logics algebraically captured by  normal lattice expansions (normal LE-logics).
\emph{MCS2010:} 03G10, 03B47, 03B45, 03B60, 03C52, 03C20.  
\end{abstract}
\tableofcontents

\section{Introduction}
This paper pertains to a line of research aimed at developing the model theory of {\em polarity-based semantics} for classes of logics algebraically captured by varieties of normal lattice expansions  in any signature (collectively referred to as {\em normal  LE-logics}). Well known instances of LE-logics  abound and have been extensively investigated (see e.g.\ \cite{grishin1983,lambek1958mathematics,goldblatt1974semantic,galatos2007residuated}). Building on results and insights developed within the theory of canonical extensions \cite{gehrke2001bounded,dunn2005canonical}, polarity-based semantics was introduced in  \cite{gehrke2006generalized} for the multiplicative fragment of the Lambek calculus, based on RS-polarities (i.e.~those polarities that dually correspond to perfect lattices). The same methodology was applied in \cite{greco2018algebraic} to define polarity-based semantics for arbitrary LE-languages in a semantic setting in which the restriction to RS-polarities is dropped.

Thanks to its generality and uniformity, the polarity-based semantics for LE-logics lends itself to support a rich mathematical theory, uniformly developed for  the whole class of LE-logics or large subclasses thereof: examples of such results are the generalized Sahlqvist theory \cite{conradie2016algorithmic}, and the uniform proof of  semantic cut elimination and finite model property for certain classes of LE-logics \cite{greco2018algebraic}, paving the way to a research program aimed at extending also other results in algebraic proof theory (e.g.~decidability via finite embeddability property, disjunction property, Craig interpolation) from substructural logics to LE-logics.

Interestingly, the polarity-based semantics has also proved suitable to support
a number of independent, pre-theoretic interpretations of the meaning of (some) LE-languages, in the same way in which Kripke semantics captures the essentials of various independent conceptual frameworks of reference for modal logic.

%
%
Specifically, in \cite{conradie2016categories,TarkPaper},  the poly-modal lattice-based logic in  the LE-language $\wedge, \vee, \top, \bot, \Box_i$ for
$i\in \mathsf{Agents}$ was given a natural interpretation as an epistemic logic of formal concepts. That is, rather than states of affairs, formulas in this language denote formal concepts. The polarity-based semantics of this language consists of structures  $\mathbb{F} = (\mathbb{P}, \{R_i\mid i\in \mathsf{Agents}\})$, referred to as {\em enriched formal contexts}, such that $\mathbb{P} = (W,U,N)$ is a polarity and $R_i\subseteq W\times U$ for each $i\in \mathsf{Agents}$.

Building on the well known interpretation of polarities in Formal Concept Analysis \cite{ganter2012formal}, each such structure can be regarded as the abstract representation of some database of objects $w\in W$ and features $u\in U$, where $wNu$ is understood as `object $w$ has feature $u$', and, if  $i$ is an agent, $wR_iu$ is understood as `object $w$ has feature $u$, according to $i$'. The classical notion of satisfaction of a formula at a state generalizes to enriched formal contexts as $w\Vdash \phi$ standing for `object $w$ is a {\em member} of category $\phi$', and $u\succ \phi$ standing for `feature $u$ {\em describes } (i.e.~is part of the intension of) category $\phi$'. For any formal concept $\phi$, the term $\Box_i\phi$ denotes the formal concept the extension of which is the set of  objects to which agent $i$ attributes all the features describing $\phi$; in symbols $\exts{\Box_i\phi}: = \{w\in W\mid \forall u(u\succ \phi \Rightarrow wR_iu)\}$. Under this interpretation, $\Box_i\phi$  intuitively denotes `concept $\phi$ according to $i$'. This interpretation is also consistent with the epistemic interpretation of well known (Sahlqvist) modal principles such as $\Box_i p\vdash p$ (classically encoding the factivity of knowledge) and $\Box_i p\vdash \Box_i\Box_i p$ (classically encoding positive introspection), relative to their first-order correspondents on enriched formal contexts. For instance, the factivity axiom above corresponds to the first order condition $R_i\subseteq N$, requiring agent $i$ to be factually correct in her attributions.

In \cite{MNPW18}, the polarity-based semantics of the LE-logic in  the language $\wedge, \vee, \top, \bot, \Box, \Diamond$ is used as a natural framework for {\em rough concepts} which unifies Formal Concept Analysis and Rough Set Theory  \cite{pawlak1998rough}. The polarity-based semantics of this language consists of structures  $\mathbb{F} = (\mathbb{P}, R_\Box, R_\Diamond)$, referred to again as {\em enriched formal contexts}, such that $\mathbb{P} = (W,U,N)$ is a polarity, $R_\Box\subseteq W\times U$, and $R_\Diamond\subseteq U\times W$ is such that $R_\Diamond^{-1} = R_\Box$. Again,  each such structure can be regarded as the abstract representation of some database of objects $w\in W$ and features $u\in U$, where $wNu$ is understood as `object $w$ has feature $u$'. However, rather than having an epistemic interpretation, $wR_\Box u$ is now understood as `object $w$ {\em demonstrably} has feature $u$'. Under this interpretation, the members of $\Box\phi$   {\em demonstrably} have all the features in the description of $\phi$, and thus $\Box\phi$ intuitively denotes the category of the {\em certified  members} of $\phi$. Moreover, $\Diamond\phi$ is the concept  described by the set of features that each member of $\phi$ {\em demonstrably} has, and thus $\Diamond\phi$ intuitively denotes the category of the {\em candidate members} of $\phi$, since  every object outside this category misses at least one feature that every member of $\phi$ demonstrably has. Also this interpretation is consistent with the interpretation of well known (Sahlqvist) modal principles such as $\Box p\vdash \Diamond p$.

Precisely the availability of these and other interpretations makes it interesting to study the expressivity of LE-logics in regard to their polarity-based semantics, and further motivates the contribution of the present paper. 
Besides its centrality in the build-up of a uniform mathematical theory of the polarity-based semantics of LE-logics, the Goldblatt-Thomason theorem provides a useful strategy to determine whether a certain elementary class of polarity-based structures  can be  captured by an LE-axiomatic principle. It is enough to show that the given class fails to reflect/be closed under one of the usual constructions to establish that no such axiomatic principle exists.

The original Goldblatt-Thomason theorem \cite{goldblatt1975axiomatic} has been extended to various classical and distributive-based logical settings which include Positive Modal Logic \cite{celani1999priestley},  coalgebraic logic \cite{kurz2007goldblatt},  graded modal logic \cite{sano2010goldblatt}, distributive substructural logics \cite{bilkova2012distributive}, {\L}ukasiewicz logic \cite{teheux2014goldblatt}, and possibility semantics for modal logic \cite{holliday2016possibility}. As to  non-distributive logical settings, recently, Goldblatt himself gave a version of it for the logic of general lattices \cite{Go17}.  Our present contribution extends this results from polarities to LE-frames (cf.\ Definition \ref{def: LE frame}). 
\paragraph{Structure of the paper.}In Section \ref{sec:preliminaries}, we collect preliminaries on LE-logics and their algebraic and polarity-based semantics; in Section \ref{sec:constructions}, we introduce the morphisms of LE-frames that correspond to complete homomorphisms of complete LE-algebras, and the relevant constructions needed for the formulation of the Goldblatt-Thomason theorem; in Section \ref{sec:enlargement}, we prove that the  ``ultrafilter extensions'' of LE-frames are  p-morphic images of some of their ultrapowers; in Section \ref{sec:GT}, the main result of this paper is stated and proved; in Section \ref{sec:applications} we use the main result to show that certain first-order conditions on LE-frames are not definable in their corresponding LE-language; in Section \ref{sec:conclusions} we collect some conclusions and further directions.

\section{Preliminaries}
\label{sec:preliminaries}
In the present section, we collect preliminaries on LE-logics. Our presentation and notation are based on \cite{greco2018algebraic}.
\subsection{Syntax and algebraic semantics of $\mathrm{LE}$-logics}
\label{subset:language:algsemantics}


Our base language is an unspecified but fixed language $\mathcal{L}_\mathrm{LE}$, to be interpreted over  lattice expansions of compatible similarity type. 
Throughout the paper, we will use the following auxiliary definition: an {\em order-type} over $n\in \mathbb{N}$ 
is an $n$-tuple $\epsilon\in \{1, \partial\}^n$. For every order type $\epsilon$, we denote its {\em opposite} order type by $\epsilon^\partial$, that is, $\epsilon^\partial_i = 1$ iff $\epsilon_i=\partial$ for every $1 \leq i \leq n$. For any lattice $\bba$, we let $\bba^1: = \bba$ and $\bba^\partial$ be the dual lattice, that is, the lattice associated with the converse partial order of $\bba$. For any order type $\varepsilon$, we let $\bba^\varepsilon: = \Pi_{i = 1}^n \bba^{\varepsilon_i}$.

The language $\mathcal{L}_\mathrm{LE}(\mathcal{F}, \mathcal{G})$ (from now on abbreviated as $\mathcal{L}_\mathrm{LE}$) takes as parameters: 1) a denumerable set of proposition letters $\mathsf{Prop}$, elements of which are denoted $p,q,r$, possibly with indexes; 2) disjoint sets of connectives $\mathcal{F}$ and $\mathcal{G}$. Each $f\in \mathcal{F}$ and $g\in \mathcal{G}$ has arity $n_f\in \mathbb{N}$ (resp.\ $n_g\in \mathbb{N}$) and is associated with some order-type $\varepsilon_f$ over $n_f$ (resp.\ $\varepsilon_g$ over $n_g$).\footnote{Unary $f$ (resp.\ $g$) will be sometimes denoted  $\Diamond$ (resp.\ $\Box$) if their order-type is 1, and ${\lhd}$ (resp.\ $\rhd$) if their order-type is $\partial$.} The terms (formulas) of $\mathcal{L}_\mathrm{LE}$ are defined recursively as follows:
\[
\phi ::= p \mid \bot \mid \top \mid \phi \wedge \phi \mid \phi \vee \phi \mid f(\overline{\phi}) \mid g(\overline{\phi})
\]
where $p \in \mathsf{Prop}$, $f \in \mathcal{F}$, $g \in \mathcal{G}$. Terms in $\mathcal{L}_\mathrm{LE}$ will be denoted either by $s,t$, or by lowercase Greek letters such as $\varphi, \psi, \gamma$ etc. 

\begin{definition}
	\label{def:LE}
	For any tuple $(\mathcal{F}, \mathcal{G})$ of disjoint sets of function symbols as above, a {\em  lattice expansion} (abbreviated as LE) is a tuple $\bba = (D, \mathcal{F}^\bbA, \mathcal{G}^\bbA)$ such that $D$ is a bounded  lattice, $\mathcal{F}^\bbA = \{f^\bbA\mid f\in \mathcal{F}\}$ and $\mathcal{G}^\bbA = \{g^\bbA\mid g\in \mathcal{G}\}$, such that every $f^\bbA\in\mathcal{F}^\bbA$ (resp.\ $g^\bbA\in\mathcal{G}^\bbA$) is an $n_f$-ary (resp.\ $n_g$-ary) operation on $\bbA$. An LE is {\em normal} if every $f^\bbA\in\mathcal{F}^\bbA$ (resp.\ $g^\bbA\in\mathcal{G}^\bbA$) preserves finite joins (resp.\ meets) in each coordinate with $\epsilon_f(i)=1$ (resp.\ $\epsilon_g(i)=1$) and reverses finite meets (resp.\ joins) in each coordinate with $\epsilon_f(i)=\partial$ (resp.\ $\epsilon_g(i)=\partial$).\footnote{\label{footnote:LE vs DLO} Normal LEs are sometimes referred to as {\em  lattices with operators} (LOs). This terminology directly derives from the setting of Boolean algebras with operators, in which operators are understood as operations which preserve finite joins in each coordinate. However, this terminology results somewhat ambiguous in the lattice setting, in which primitive operations are typically maps which are operators if seen as $\bbA^\epsilon\to \bbA^\eta$ for some order-type $\epsilon$ on $n$ and some order-type $\eta\in \{1, \partial\}$. Rather than speaking of  lattices with $(\varepsilon, \eta)$-operators, we then speak of normal LEs.} Let $\mathbb{LE}$ be the class of LEs. Sometimes we will refer to certain LEs as $\mathcal{L}_\mathrm{LE}$-algebras when we wish to emphasize that these algebras have a compatible signature with the logical language we have fixed.
\end{definition}

In the remainder of the paper, 
we will abuse notation and write e.g.\ $f$ for $f^\bbA$.
Henceforth, every LE is assumed to be normal; hence the adjective `normal' will be typically dropped. The class of all LEs is equational, and can be axiomatized by the usual  lattice identities and the following equations for any $f\in \mathcal{F}$ (resp.\ $g\in \mathcal{G}$) and $1\leq i\leq n_f$ (resp.\ for each $1\leq j\leq n_g$):
\begin{itemize}
	\item if $\varepsilon_f(i) = 1$, then $f(p_1,\ldots, p\vee q,\ldots,p_{n_f}) = f(p_1,\ldots, p,\ldots,p_{n_f})\vee f(p_1,\ldots, q,\ldots,p_{n_f})$ and $f(p_1,\ldots, \bot,\ldots,p_{n_f}) = \bot$,
	\item if $\varepsilon_f(i) = \partial$, then $f(p_1,\ldots, p\wedge q,\ldots,p_{n_f}) = f(p_1,\ldots, p,\ldots,p_{n_f})\vee f(p_1,\ldots, q,\ldots,p_{n_f})$ and $f(p_1,\ldots, \top,\ldots,p_{n_f}) = \bot$,
	\item if $\varepsilon_g(j) = 1$, then $g(p_1,\ldots, p\wedge q,\ldots,p_{n_g}) = g(p_1,\ldots, p,\ldots,p_{n_g})\wedge g(p_1,\ldots, q,\ldots,p_{n_g})$ and $g(p_1,\ldots, \top,\ldots,p_{n_g}) = \top$,
	\item if $\varepsilon_g(j) = \partial$, then $g(p_1,\ldots, p\vee q,\ldots,p_{n_g}) = g(p_1,\ldots, p,\ldots,p_{n_g})\wedge g(p_1,\ldots, q,\ldots,p_{n_g})$ and $g(p_1,\ldots, \bot,\ldots,p_{n_g}) = \top$.
\end{itemize}
Each language $\mathcal{L}_\mathrm{LE}$ is interpreted in the appropriate class of LEs. In particular, for every LE $\bba$, each operation $f^\bba\in \mathcal{F}^\bbA$ (resp.\ $g^\bba\in \mathcal{G}^\bbA$) is finitely join-preserving (resp.\ meet-preserving) in each coordinate when regarded as a map $f^\bba: \bba^{\varepsilon_f}\to \bba$ (resp.\ $g^\bba: \bba^{\varepsilon_g}\to \bba$).

\begin{definition}\label{def:can:ext}
	The \emph{canonical extension} of a BL (bounded lattice) $L$ is a complete lattice $L^\delta$ containing $L$ as a sublattice, such that:
	\begin{enumerate}
		\item \emph{(denseness)} every element of $L^\delta$ can be expressed both as a join of meets and as a meet of joins of elements from $L$;
		\item \emph{(compactness)} for all $S,T \subseteq L$, if $\bigwedge S \leq \bigvee T$ in $L^\delta$, then $\bigwedge F \leq \bigvee G$ for some finite sets $F \subseteq S$ and $G\subseteq T$.
	\end{enumerate}
\end{definition}

It is well known that the canonical extension of a BL $L$ is unique up to isomorphism fixing $L$ (cf.\ e.g.\ \cite[Section 2.2]{GNV05}), and that the canonical extension of a BL is a \emph{perfect} BL, i.e.\ a complete  lattice which is completely join-generated by its completely join-irreducible elements and completely meet-generated by its completely meet-irreducible elements  (cf.\ e.g.\ \cite[Definition 2.14]{GNV05})\label{canext bdl is perfect}. The canonical extension of an
$\mathcal{L}_\mathrm{LE}$-algebra $\bbA = (L, \mathcal{F}^\bbA, \mathcal{G}^\bbA)$ is the perfect  $\mathcal{L}_\mathrm{LE}$-algebra
$\bbA^\delta: = (L^\delta, \mathcal{F}^{\bbA^\delta}, \mathcal{G}^{\bbA^\delta})$ such that $f^{\bbA^\delta}$ and $g^{\bbA^\delta}$ are defined as the
$\sigma$-extension of $f^{\bbA}$ and as the $\pi$-extension of $g^{\bbA}$ respectively, for all $f\in \mathcal{F}$ and $g\in \mathcal{G}$ (cf.\ \cite{sofronie2000duality, sofronie2000duality2}).

\medskip

The generic LE-logic is not equivalent to a sentential logic. Hence the consequence relation of these logics cannot be uniformly captured in terms of theorems, but rather in terms of sequents, which motivates the following definition:
\begin{definition}
	\label{def:LE:logic:general}
	For any language $\mathcal{L}_\mathrm{LE} = \mathcal{L}_\mathrm{LE}(\mathcal{F}, \mathcal{G})$, the {\em basic}, or {\em minimal} $\mathcal{L}_\mathrm{LE}$-{\em logic} is a set of sequents $\phi\vdash\psi$, with $\phi,\psi\in\mathcal{L}_\mathrm{LE}$, which contains the following axioms:
	\begin{itemize}
		\item Sequents for lattice operations:\footnote{In what follows we will use the turnstile symbol $\vdash$ both as sequent separator and also as the consequence relation of the logic.}
		\begin{align*}
		&p\vdash p, && \bot\vdash p, && p\vdash \top, &p\vdash p\vee q \\
		& q\vdash p\vee q, && p\wedge q\vdash p, && p\wedge q\vdash q, &
		\end{align*}
		\item Sequents for additional connectives:
		\begin{align*}
		& f(p_1,\ldots, \bot,\ldots,p_{n_f}) \vdash \bot,~\mathrm{for}~ \varepsilon_f(i) = 1,\\
		& f(p_1,\ldots, \top,\ldots,p_{n_f}) \vdash \bot,~\mathrm{for}~ \varepsilon_f(i) = \partial,\\
		&\top\vdash g(p_1,\ldots, \top,\ldots,p_{n_g}),~\mathrm{for}~ \varepsilon_g(i) = 1,\\
		&\top\vdash g(p_1,\ldots, \bot,\ldots,p_{n_g}),~\mathrm{for}~ \varepsilon_g(i) = \partial,\\
		&f(p_1,\ldots, p\vee q,\ldots,p_{n_f}) \vdash f(p_1,\ldots, p,\ldots,p_{n_f})\vee f(p_1,\ldots, q,\ldots,p_{n_f}),~\mathrm{for}~ \varepsilon_f(i) = 1,\\
		&f(p_1,\ldots, p\wedge q,\ldots,p_{n_f}) \vdash f(p_1,\ldots, p,\ldots,p_{n_f})\vee f(p_1,\ldots, q,\ldots,p_{n_f}),~\mathrm{for}~ \varepsilon_f(i) = \partial,\\
		& g(p_1,\ldots, p,\ldots,p_{n_g})\wedge g(p_1,\ldots, q,\ldots,p_{n_g})\vdash g(p_1,\ldots, p\wedge q,\ldots,p_{n_g}),~\mathrm{for}~ \varepsilon_g(i) = 1,\\
		& g(p_1,\ldots, p,\ldots,p_{n_g})\wedge g(p_1,\ldots, q,\ldots,p_{n_g})\vdash g(p_1,\ldots, p\vee q,\ldots,p_{n_g}),~\mathrm{for}~ \varepsilon_g(i) = \partial,
		\end{align*}
	\end{itemize}
	and is closed under the following inference rules:
	\begin{displaymath}
	\frac{\phi\vdash \chi\quad \chi\vdash \psi}{\phi\vdash \psi}
	\quad
	\frac{\phi\vdash \psi}{\phi(\chi/p)\vdash\psi(\chi/p)}
	\quad
	\frac{\chi\vdash\phi\quad \chi\vdash\psi}{\chi\vdash \phi\wedge\psi}
	\quad
	\frac{\phi\vdash\chi\quad \psi\vdash\chi}{\phi\vee\psi\vdash\chi}
	\end{displaymath}
	\begin{displaymath}
	\frac{\phi\vdash\psi}{f(p_1,\ldots,\phi,\ldots,p_n)\vdash f(p_1,\ldots,\psi,\ldots,p_n)}{~(\varepsilon_f(i) = 1)}
	\end{displaymath}
	\begin{displaymath}
	\frac{\phi\vdash\psi}{f(p_1,\ldots,\psi,\ldots,p_n)\vdash f(p_1,\ldots,\phi,\ldots,p_n)}{~(\varepsilon_f(i) = \partial)}
	\end{displaymath}
	\begin{displaymath}
	\frac{\phi\vdash\psi}{g(p_1,\ldots,\phi,\ldots,p_n)\vdash g(p_1,\ldots,\psi,\ldots,p_n)}{~(\varepsilon_g(i) = 1)}
	\end{displaymath}
	\begin{displaymath}
	\frac{\phi\vdash\psi}{g(p_1,\ldots,\psi,\ldots,p_n)\vdash g(p_1,\ldots,\phi,\ldots,p_n)}{~(\varepsilon_g(i) = \partial)}.
	\end{displaymath}
	The minimal $\mathcal{L}_\mathrm{LE}$-logic is denoted $\mathbb{L}_\mathrm{LE}$. By an {\em $\mathrm{LE}$-logic} we understand any axiomatic extension of $\mathbb{L}_\mathrm{LE}$ in the language $\mathcal{L}_{\mathrm{LE}}$. 
\end{definition}

For every LE $\bba$, the symbol $\vdash$ is interpreted as the lattice order $\leq$. A sequent $\phi\vdash\psi$ is valid in $\bba$ if $h(\phi)\leq h(\psi)$ for every homomorphism $h$ from the $\mathcal{L}_\mathrm{LE}$-algebra of formulas over $\mathsf{Prop}$ to $\bba$. The notation $\mathbb{LE}\models\phi\vdash\psi$ indicates that $\phi\vdash\psi$ is valid in every LE. Then, by means of a routine Lindenbaum-Tarski construction, it can be shown that the minimal LE-logic $\mathbb{L}_\mathrm{LE}$ is sound and complete with respect to its correspondent class of algebras $\mathbb{LE}$, i.e.\ that any sequent $\phi\vdash\psi$ is provable in $\mathbb{L}_\mathrm{LE}$ iff $\mathbb{LE}\models\phi\vdash\psi$. 

\subsection{$\mathrm{LE}$-frames and their complex algebras}
From now on, we fix an arbitrary normal LE-signature $\mathcal{L} = \mathcal{L}(\mathcal{F}, \mathcal{G})$.


\subsubsection{Notation}\label{ssec:notation}
For any sets $A, B$ and any relation $S \subseteq A \times B$, we let, for any $A' \subseteq A$ and $B' \subseteq B$,
$$S^\uparrow[A'] := \{b \in B\mid  \forall a(a \in A' \Rightarrow a\ S\ b ) \} \,\, \mathrm{and}\,\, S^\downarrow[B'] := \{a \in A \mid \forall b(b \in B' \Rightarrow a\ S\ b)  \}.$$
For all sets $A, B_1,\ldots B_n,$ and any relation $S \subseteq A \times B_1\times \cdots\times B_n$, for any $\overline{C}: = (C_1,\ldots, C_n)$ where $C_i\subseteq B_i$ and $1\leq i\leq n$ 
we let, for all $A'$,
\begin{equation}\label{eq:notation bari}
\bari{C}:  = (C_1,\ldots,C_{i-1}, C_{i+1},\ldots, C_n)
\end{equation}
\begin{equation}\label{eq:notation baridown}
\bari{C}_{A'}: = (C_1\ldots,C_{i-1}, A', C_{i+1},\ldots, C_n)
\end{equation}
\begin{equation}\label{eq:notation baridownandmore}
\overline{C}^{j,i}_{ A'}: = \overline{Y}^{j}\text{ where }Y\text{ is the generic element of }\bari{C}_{A'}
\end{equation}
that is, $\overline{C}^{j,i}_{A'}$ is the sequence obtained from $\overline{C}$ by replacing $C_i$ by $A'$ and removing the $j$-th coordinate.  
When $C_i: = \{c_i\}$ and $A': =\{a'\} $, we will write $\overline{c}$ for $\overline{\{c\}}$,  and $\bari{c}$ for $\bari{\{c\}}$, and $\bari{c}_{a'}$ for $\bari{\{c\}}_{\{a'\}}$.
We also let:
\begin{enumerate}
	\item $S^{(0)}[\overline{C}] := \{a \in A\mid  \forall \overline{b}(\overline{b}\in \overline{C} \Rightarrow a\ S\ \overline{b} ) \} .$
	
	\item $S_i \subseteq B_i \times B_1 \times \cdots B_{i-1} \times A \times B_{i + 1} \times \cdots\times B_n$ be defined by
	\[(b_i, \overline{c}_{a}^{\, i})\in S_i \mbox{ iff } (a,\overline{c})\in S.\]
	
	\item $S^{(i)}[A', \bari{C}] := S_i^{(0)}[\overline{C}^{\, i}_{A'}]$.
\end{enumerate}

\begin{lemma}[cf.\ \cite{greco2018algebraic} Lemma 15]\label{lem:monochrome}
	If  $S \subseteq A \times B_1\times \cdots\times B_n$ and $\overline{C}$ is as above, then for any $1\leq i \leq n$,
	\begin{equation}
	\label{composition}
	C_i \subseteq S^{(i)}[S^{(0)}[\overline{C}], \bari{C}].
	\end{equation}
\end{lemma}

\subsubsection{$\mathrm{LE}$-frames}

\begin{definition}[Polarity]\label{def:polarity}
	A {\em polarity} is a structure $\mathbb{W} = (W, U, N)$ where $W$ and $U$ are sets and $N$ is a binary relation from $W$ to $U$.
\end{definition}
If $\mathbb{L}$ is a lattice, then $\mathbb{W}_L = (L, L, \leq)$ is a polarity. Conversely, for any polarity $\mathbb{W}$, we let $\mathbb{W}^+$ denote the concept lattice associated with $\mathbb{W}$. Any $a\in\mathbb{W}^+$ can be represented as a tuple $(\exts{a},\ints{a})$ such that $\ints{a}=\exts{a}^\uparrow$ and $\exts{a}=\ints{a}^\downarrow$, where  for every $X \subseteq W$ and $Y \subseteq U$, $X^\uparrow$ and $Y^\downarrow$ are abbreviations for $N^\uparrow[X]$ and $N^\downarrow[Y]$ respectively. As is well-known, $\mathbb{W}^+$ is isomorphic to the complete sub $\bigcap$-semilattice of the Galois-stable sets of the closure operator $\gamma_N: \mathcal{P}(W) \rightarrow \mathcal{P}(W)$ defined by the assignment  $X\mapsto  X^{\uparrow\downarrow}$. Hence,  $\mathbb{W}^+$ is a complete lattice, in which $\bigvee S := \gamma_N(\bigcup S)$ for any $S \subseteq \gamma_N[\mathcal{P}(W)]$. Moreover, $\mathbb{W}^+$ can be equivalently obtained as the dual lattice of the Galois-stable sets of the closure operator $\gamma'_N: \mathcal{P}(U) \rightarrow \mathcal{P}(U)$ defined by the assignment  $Y\mapsto  Y^{\downarrow\uparrow}$.


From now on, we 
focus on $\mathcal{L}$-algebras $\mathbb{A} = (L, \wedge, \vee, \bot, \top, \mathcal{F}, \mathcal{G})$.




\begin{definition}\label{def: LE frame}
	An {\em $\mathcal{L}$-frame} is a tuple $\mathbb{F} = (\mathbb{W}, \mathcal{R}_{\mathcal{F}}, \mathcal{R}_{\mathcal{G}})$ such that $\mathbb{W} = (W, U,  N)$ is a polarity, $\mathcal{R}_{\mathcal{F}} = \{R_f\mid f\in \mathcal{F}\}$, and $\mathcal{R}_{\mathcal{G}} = \{R_g\mid g\in \mathcal{G}\}$ such that  for each $f\in \mathcal{F}$ and $g\in \mathcal{G}$, the symbols $R_f$ and  $R_g$ respectively denote $(n_f+1)$-ary and $(n_g+1)$-ary relations on $\mathbb{W}$,
	\begin{equation}
	R_f \subseteq U \times W^{\epsilon_f}    \ \mbox{ and }\ R_g \subseteq W \times U^{\epsilon_g},
	\end{equation}
	where for any order type $\epsilon$ on $n$, we let $W^{\epsilon} : = \prod_{i = 1}^{n}W^{\epsilon(i)}$ and $U^{\epsilon} : = \prod_{i = 1}^{n}U^{\epsilon(i)}$, where for all $1 \leq i \leq n$,
	
	\begin{center}
		\begin{tabular}{ll}
			$W^{\epsilon(i)} = \begin{cases}
			W &\mbox{ if } \epsilon(i) = 1\\
			U &\mbox{ if } \epsilon(i) = \partial
			\end{cases}\quad$
			&
			$U^{\epsilon(i)} = \begin{cases}
			U& \mbox{ if } \epsilon(i) = 1,\\
			W & \mbox{ if } \epsilon(i) = \partial.
			\end{cases}$
		\end{tabular}
	\end{center}
	
	In addition, we assume that  the following sets are Galois-stable (from now on abbreviated as {\em stable}) for all $w_0 \in W$, $u_0 \in U $, $\overline{w} \in W^{\epsilon_f}$, and $\overline{u} \in U^{\epsilon_g}$:
	
	\begin{equation}\label{eq:compa1}
		R_f^{(0)}[\overline{w}]\text{ and }R_f^{(i)}[u_0, \bari{w}] 
		\end{equation}
		
	 \begin{equation}\label{eq:compa2}
		R_g^{(0)}[\overline{u}]\text{ and }R_g^{(i)}[w_0, \bari{u}]
		\end{equation}

\end{definition}


In what follows, for any order type $\epsilon$ on $n$,  we let
$$ W^{\epsilon} \supseteq \overline{X}:= (X^{\epsilon(1)}, \ldots, X^{\epsilon(n)}),$$
where $X^{\epsilon(i)} \subseteq W^{\epsilon(i)}$ for all $1 \leq i \leq n$, and let
$$U^{\epsilon} \supseteq \overline{Y}:= (Y^{\epsilon(1)}, \ldots, Y^{\epsilon(n)}),$$
where $Y^{\epsilon(i)} \subseteq U^{\epsilon(i)}$ for all $1 \leq i \leq n$. Moreover,
we let $\overline{X}^i$, $\overline{X}^i_Z$, $\overline{Y}^i$ and $\overline{X}^i_Z$ be defined as in Subsection \ref{ssec:notation}.


\begin{lemma}[cf.\ \cite{greco2018algebraic} Lemma 18]\label{lem:Ri closed}
	For any $\mathcal{L}$-frame $\mathbb{F} = (\mathbb{W}, \mathcal{R}_{\mathcal{F}}, \mathcal{R}_{\mathcal{G}})$, any $f \in \mathcal{F}$ and $g \in \mathcal{G},$ \begin{enumerate}
		\item 
		if $Y_0 \subseteq U$,
		then
		$R_f^{(0)}[\overline{X}] $ and $R_f^{(i)}[Y_0, \bari{X}] $ are stable sets for all $1 \leq i \leq n_f$;
		
		\item 
		if $X_0 \subseteq W$,
		then
		$R_g^{(0)}[\overline{Y}] $ and $R_g^{(i)}[X_0, \bari{Y}] $ are stable sets for all $1 \leq i \leq n_g$.
	\end{enumerate}
\end{lemma}

The following lemma gives equivalent conditions to \eqref{eq:compa1} and \eqref{eq:compa2}. We make use of notation introduced in \eqref{eq:notation bari}, \eqref{eq:notation baridown}, \eqref{eq:notation baridownandmore}. To simplify the notation we identify $\gamma_N$ and $\gamma'_N$.
\begin{lemma}\label{lem:two definitions of compatibility}
	Let $\mathbb{W}=(W,U,N)$ be a polarity and $\epsilon$ be an order type on $n$.
	\begin{enumerate}[i]
		\item For any $R\subseteq U\times W^{\epsilon}$ and any $0\leq i\leq n$, the following are equivalent: \begin{enumerate} 
			\item $R^{(i)}[\bari{X}]$ is stable for every $\overline{X}\subseteq U\times W^{\epsilon}$.
			\item $R^{(j)}[\overline{X}^{j}]=R^{(j)}[\overline{X}^{j,i}_{\gamma_N(X_i)}]$ for every $\overline{X}\subseteq W^{\epsilon}$ and $j\neq i$.
		\end{enumerate}
		\item  For any $R\subseteq W\times U^{\epsilon}$ and any $0\leq i\leq n$, the following are equivalent: \begin{enumerate}
		\item $R^{(i)}[\bari{Y}]$ is stable for every $\overline{Y}\subseteq W\times  U^{\epsilon}$.
		\item $R^{(j)}[\overline{Y}]=R^{(j)}[\overline{Y}^{j,i}_{\gamma_N(Y_i)}]$ for every $\overline{Y}\subseteq W\times U^{\epsilon}$ and $j\neq i$.
	\end{enumerate}
	\end{enumerate}
\end{lemma}
\begin{proof}
	i. By definition, for any $i,j$ and $\overline{X}$, \begin{equation}\label{eq:keypointofequivalencedefinitions}
	X_j\subseteq R^{(j)}[\overline{X}^{j}]\iff X_i\subseteq R^{(i)}[\bari{X}].
	\end{equation}Let us assume that $R^{(i)}[\overline{X}]$ is stable for every $\overline{X}\subseteq U\times W^{\epsilon}$ and show that $R^{(j)}[\overline{X}^{j}]\subseteq R^{(j)}[\overline{X}^{j,i}_{\gamma_N(X_i)}]$, the converse inclusion following from the antitonicity of $R^{(j)}$:
	\begin{align*}
	&	\quad R^{(j)}[\overline{X}^{j}]\subseteq  R^{(j)}[\overline{X}^{j}]\\
	\iff & \quad  X_i\subseteq R^{(i)}[\overline{X}^{i,j}_{R^{(j)}[\overline{X}]}]\tag{by \ref{eq:keypointofequivalencedefinitions}}\\
	\iff & \quad  \gamma_N(X_i)\subseteq R^{(i)}[\overline{X}^{i,j}_{R^{(j)}[\overline{X}]}]\tag{$R^{(i)}[\overline{X}^{i,j}_{R^{(j)}[\overline{X}]}]$ is stable by assumption}\\
	\iff & R^{(j)}[\overline{X}^{j}]\subseteq R^{(j)}[\overline{X}^{j,i}_{\gamma_N(X_i)}]\tag{by \ref{eq:keypointofequivalencedefinitions}}.
	\end{align*}
	Now assume that $R^{(j)}[\overline{X}^{j}]=R^{(j)}[\overline{X}^{j,i}_{\gamma_N(X_i)}]$ for every $\overline{X}\subseteq U\times W^{\epsilon}$ and $j\neq i$. We want to show that $\gamma_N(R^{(i)}[\overline{X}^{i}])\subseteq R^{(i)}[\overline{X}^{i}]$:
	\begin{align*}
	&	\quad R^{(i)}[\overline{X}^{i}]\subseteq  R^{(i)}[\overline{X}^{i}]\\
	\iff & \quad  X_j\subseteq R^{(j)}[\overline{X}^{j,i}_{R^{(i)}[\overline{X}^{i}]}]\tag{by \ref{eq:keypointofequivalencedefinitions}}\\
	\iff & \quad  X_j\subseteq R^{(j)}[\overline{X}^{j,i}_{\gamma_N(R^{(i)}[\overline{X}^{i}])}]\tag{$R^{(j)}[\overline{X}^{j,i}_{R^{(i)}[\overline{X}^{i}]}]=R^{(j)}[\overline{X}^{j,i}_{\gamma_N(R^{(i)}[\overline{X}^{i}])}]$ by assumption}\\
	\iff & \gamma_N(R^{(i)}[\overline{X}^{i}])\subseteq R^{(i)}[\overline{X}^{i}]\tag{by \ref{eq:keypointofequivalencedefinitions}}.
	\end{align*}
The proof of (ii) follows verbatim.
\end{proof}

\begin{remark}
	In case $R\subseteq U\times W$, the above lemma states that $R^{(0)}[X]$ is stable for every $X\subseteq W$ if and only if $R^{(1)}[Y]=R^{(1)}[Y^{\downarrow\uparrow}]$ for any $Y\subseteq U$, and  $R^{(1)}_f[Y]$ is stable for every $Y\subseteq U$ if and only if $R^{(0)}[X]=R^{(0)}[X^{\uparrow\downarrow}]$ for any $X\subseteq W$. Hence the lemma above gives an equivalent reformulation of the definition of compatibility in \cite{moshier2016relational} (see also Lemma 1.4 therein).
\end{remark}

\subsubsection{Complex algebras of $\mathrm{LE}$-frames}
Given a polarity $\mathbb{W}$ and $a\in\mathbb{W}^+$, $\epsilon:\{1,\ldots,n\}\to\{1,\partial\}$, for every $1\leq i\leq n$ we denote 
\[
\exts{a}^{\epsilon(i)} = \begin{cases}
\exts{a} &\mbox{ if }\epsilon(i)= 1,\\  \ints{a}&\mbox{ if }\epsilon(i)= \partial\\
\end{cases}
\]
and
\[
\ints{a}^{\epsilon(i)} = \begin{cases}
\ints{a} &\mbox{ if }\epsilon(i)= 1,\\  \exts{a}&\mbox{ if }\epsilon(i)= \partial.\\
\end{cases}
\]
\begin{definition} \label{def:complexalg}
	The {\em complex algebra} of  an $\mathcal{L}$-frame $\mathbb{F} = (\mathbb{W}, \mathcal{R}_{\mathcal{F}}, \mathcal{R}_{\mathcal{G}})$ is the algebra
	$$\mathbb{F}^+ = (\mathbb{L}, \{f_{R_f} \mid f \in \mathcal{F}\}, \{g_{R_g} \mid g \in \mathcal{G}\}),$$
	where $\mathbb{L} := \mathbb{W}^+$ (cf.~Definition \ref{def:polarity}), and for all $f \in \mathcal{F}$ and all $g \in \mathcal{G}$,
	we let
	\begin{enumerate}
		\item $ f_{R_f}: \mathbb{L}^{n_f}\to \mathbb{L}$ be defined by the assignment $f_{R_f}(\overline{a}) = ((R_f^{(0)}[\overline{\exts{a}}^{\epsilon_f}])^\downarrow,R_f^{(0)}[\overline{\exts{a}}^{\epsilon_f}])$;
		\item $g_{R_g}: \mathbb{L}^{n_g}\to \mathbb{L}$ be defined by the assignment  $g_{R_g}(\overline{a}) = (R_g^{(0)}[\overline{\exts{a}}^{\epsilon_g}],(R_g^{(0)}[\overline{\exts{a}}^{\epsilon_g}])^\uparrow)$.  
	\end{enumerate}
	  	
\end{definition}

\begin{prop}[cf.\ \cite{greco2018algebraic} Proposition 21]\label{prop:F plus is L star algebra}
	If $\mathbb{F}$ is an $\mathcal{L}$-frame,  then $\mathbb{F}^+$ is a complete $\mathcal{L}$-algebra.
\end{prop}

\subsection{Algebraic and relational models}
Specializing the usual interpretation of $\mathcal{L}$-formulas into $\mathcal{L}$-algebras to complex algebras of $\mathcal{L}$-frames yields the following.  
\begin{definition}\label{def:models and val}
	For any $\mathcal{L}$-frame $\mathbb{F}$ and any $V:\mathsf{Prop}\to\mathbb{F}^+$, the unique homomorphic extension of $V$, denoted also $V:\mathcal{L}\to\mathbb{F}^+$, is defined recursively as follows:
\begin{center}
	\begin{tabular}{r c l}
		$V(p)$ & $ = $ & $(\exts{V(p)}, \ints{V(p)})$\\
		$V(\top)$ & $ = $ & $(W, W^{\uparrow})$ \\
		$V(\bot)$ & $ = $ & $(U^{\downarrow}, U)$\\
		$V(\phi\wedge\psi)$ & $ = $ & $(\exts{V(\phi)}\cap \exts{V(\psi)}, (\exts{V(\phi)}\cap \exts{V(\psi)})^{\uparrow})$\\
		$V(\phi\vee\psi)$ & $ = $ & $((\ints{V(\phi)}\cap \ints{V(\psi)})^{\downarrow}, \ints{V(\phi)}\cap \ints{V(\psi)})$\\
		$V(g(\overline{\phi}))$ & $ = $ & $(R_g^{(0)}[\overline{\ints{V(\phi)}^{\epsilon_g}}], (R_g^{(0)}[\overline{\ints{V(\phi)}^{\epsilon_g}}])^{\uparrow})$\\
		$V(f(\overline{\phi}))$ & $ = $ & $((R_f^{(0)}[\overline{\exts{V(\phi)}^{\epsilon_f}}])^{\downarrow}, R_f^{(0)}[\overline{\exts{V(\phi)}^{\epsilon_f}}])$\\
	\end{tabular}
\end{center}
As usual for any $\mathcal{L}$-sequent $\varphi\vdash\psi$, we say that $\mathbb{F}^+,V\models\varphi\vdash\psi$ if $V(\varphi)\leq V(\psi)$, i.e.\ $\exts{V(\varphi)}\subseteq\exts{V(\psi)}$ or equivalently $\ints{V(\psi)}\subseteq\ints{V(\varphi)}$. The $\mathcal{L}$-sequent $\varphi\vdash\psi$ is \emph{valid} on $\mathbb{F}^+$, in symbols $\mathbb{F}^+\models\varphi\vdash\psi$, if $\mathbb{F}^+,V\models\varphi\vdash\psi$ for every valuation $V:\mathsf{Prop}\to\mathbb{F}^+$.
\end{definition}
In the remainder of the paper we will abbreviate $\exts{V(\psi)}$ as $\exts{\psi}$ and $\ints{V(\psi)}$ as $\ints{\psi}$ when $V$ is clear from the context. 

\begin{definition}
	An $\mathcal{L}$\emph{-model} is a tuple $\mathbb{M}=(\mathbb{F},V)$ such that $\mathbb{F}$ is an $\mathcal{L}$-frame and $V:\mathsf{Prop}\to\mathbb{F}^+$ is a valuation.
\end{definition}
Unraveling the recursive definition of the unique homomorphic extension of a given valuation yields the following:
\begin{definition}
For any $\mathcal{L}$-model $\mathbb{M}=(\mathbb{F},V)$, the \emph{satisfaction} and \emph{co-satisfaction} relations, ${\Vdash}\subseteq W\times \mathcal{L}$ and ${\succ}\subseteq U\times\mathcal{L}$, are defined by simultaneous recursion as follows:
\begin{center}
	\begin{tabular}{llll}
		$\mathbb{M}, w \Vdash p$ & iff & $w\in \exts{V(p)}$  & \\
		$\mathbb{M}, u \succ p$ & iff & $u\in \ints{V(p)}$  &\\
		$\mathbb{M}, w \Vdash\top$ &  & always \\
		$\mathbb{M}, u \succ \top$ & iff &   $w N u$ for all $w\in W$\\
		$\mathbb{M}, u \succ  \bot$ &  & always \\
		$\mathbb{M}, w \Vdash \bot $ & iff & $w N u$ for all $u\in U$\\
		$\mathbb{M}, w \Vdash \phi\wedge \psi$ & iff & $\mathbb{M}, w \Vdash \phi$ and $\mathbb{M}, w \Vdash  \psi$ & \\
		$\mathbb{M}, u \succ \phi\wedge \psi$ & iff & for all $w\in W$, if $\mathbb{M}, w \Vdash \phi\wedge \psi$, then $w N u$\\
		$\mathbb{M}, u \succ \phi\vee \psi$ & iff &  $\mathbb{M}, u \succ \phi$ and $\mathbb{M}, u \succ  \psi$ &\\
		$\mathbb{M}, w \Vdash \phi\vee \psi$ & iff & for all $u\in U$, if $\mathbb{M}, u \succ \phi\vee \psi$, then $w N u$  & \\
		$\mathbb{M}, w \Vdash g(\overline{\phi})$ & iff & for all $\overline{u}\in U^{\epsilon_g}$, if $\mathbb{M}, u^{\epsilon_g(i)} \succ^{\epsilon_g(i)} \phi$ for every $1\leq i\leq n_g$, then $R_g (w, \overline{u})$& \\
		$\mathbb{M}, u \succ g(\overline{\phi})$ & iff & for all $w\in w$, if $\mathbb{M}, w \Vdash g(\overline{\phi})$, then $w Nu$.\\
		$\mathbb{M}, u \succ f(\overline{\phi})$ & iff & for all $\overline{w}\in U^{\epsilon_f}$, if $\mathbb{M}, w^{\epsilon_f(i)} \Vdash^{\epsilon_f(i)} \phi$ for every $1\leq i\leq n_f$, then $R_f (u, \overline{w})$& \\
		$\mathbb{M}, w \Vdash f(\overline{\phi})$ & iff & for all $u\in U$, if $\mathbb{M}, u \succ g(\overline{\phi})$, then $w Nu$.\\
		%
	\end{tabular}
\end{center}

In the table above, $u^{\epsilon_g(i)}\in U$ if $\epsilon_g(i) = 1$ and $u^{\epsilon_g(i)}\in W$ if $\epsilon_g(i) = \partial$; likewise,
$w^{\epsilon_f(i)}\in W$ if $\epsilon_f(i) = 1$ and $w^{\epsilon_f(i)}\in U$ if $\epsilon_f(i) = \partial$. Moreover, $\succ^{\epsilon_g(i)} = \succ$ if $\epsilon_g(i) = 1$ and $\succ^{\epsilon_g(i)} = \Vdash$ if $\epsilon_g(i) = \partial$; likewise, $\Vdash^{\epsilon_f(i)} = \Vdash$ if $\epsilon_f(i) = 1$ and $\Vdash^{\epsilon_f(i)} = \succ$ if $\epsilon_g(i) = \partial$.

Moreover, for any $\mathcal{L}$-sequent $\varphi\vdash\psi$, we write 
\begin{center}
	\begin{tabular}{llll}
		$\mathbb{M}\models \varphi\vdash\psi$ & iff & for every $w\in W$ if $\mathbb{M},w\Vdash\varphi$ then $\mathbb{M},w\Vdash\psi$  & \\
		& iff & for every $u\in U$ if $\mathbb{M},u\succ\psi$ then $\mathbb{M},u\succ\varphi$  & \\
	\end{tabular}
\end{center}
The $\mathcal{L}$-sequent $\varphi\vdash\psi$ is \emph{valid} on $\mathbb{F}$, in symbols $\mathbb{F}\models\varphi\vdash\psi$, if $\mathbb{F},V\models\varphi\vdash\psi$ for every valuation $V:\mathsf{Prop}\to\mathbb{F}^+$.
\end{definition}
The following proposition can be straightforwardly verified.
\begin{prop}\label{prop:frames to algebras}
	For every $\mathcal{L}$-frame $\mathbb{F}$ and every $\mathcal{L}$-sequent $\varphi\vdash\psi$, $$\mathbb{F}\models\varphi\vdash\psi\quad\iff\quad\mathbb{F}^+\models\varphi\vdash\psi.$$
\end{prop}

As in the Boolean case, each $\mathcal{L}$-model $\mathbb{M}$  can be seen as a two-sorted first-order structure. Accordingly, we define the correspondence language as follows.

Let $L_1$ be the two-sorted first-order language with equality  built over the denumerable and disjoint sets of individual variables $W$ and $U$, with binary relation symbol $N$, and $(n_f+1)$-ary relation symbols $R_{f}$ for each $f\in \mathcal{F}$, and $(n_g+1)$-ary relation symbols $R_{g}$ for each $g\in \mathcal{G}$ and two unary predicate symbols $P_{\exts{p}}, P_{\ints{p}}$ for each propositional variable $p \in \mathsf{Prop}$.\footnote{The intended interpretation links $P_{\exts{p}}$ and $ P_{\ints{p}}$ in the way suggested by the definition of $\mathcal{L}$-valuations. Indeed, every $p \in \mathsf{Prop}$ is mapped to a pair $(\exts{p}, \ints{p})$ of Galois-stable sets as indicated in Definition \ref{def:models and val}. Accordingly, the interpretation of pairs $(P_{\exts{p}}, P_{\ints{p}})$ of predicate symbols is restricted to such pairs of Galois-stable sets, and hence the interpretation of universal second-order quantification is also restricted to range over such sets.}


\begin{definition}\label{def:st}The \emph{standard translation} of $\mathcal{L}$ into $L_1$ is given by the following recursion:
	\begin{center}
	{\small{
			\begin{tabular}{ll}
				$\mathrm{ST}_x(\bot) := \forall y(xNy)$ & $\mathrm{ST}_y(\bot) := y  = y$ \\
				$\mathrm{ST}_x(\top) := x = x$ & $\mathrm{ST}_y(\top) := \forall x(xNy)$\\
				$\mathrm{ST}_x(p) := P_{\exts{p}}(x)$& $\mathrm{ST}_y(p) :=  P_{\ints{p}}(y)$\\
				%
				%
				$\mathrm{ST}_x(\phi \vee \psi) := \forall y[\mathrm{ST}_y(\phi\vee \psi)\rightarrow xNy]$& $\mathrm{ST}_y(\phi \vee \psi) := \mathrm{ST}_y(\phi) \wedge \mathrm{ST}_y(\psi)$\\
				$\mathrm{ST}_x(\phi \wedge \psi) := \mathrm{ST}_x(\phi) \wedge \mathrm{ST}_x(\psi)$& $\mathrm{ST}_y(\phi \wedge \psi) := \forall x[\mathrm{ST}_x(\phi\wedge \psi)\rightarrow xNy]$\\
				$\mathrm{ST}_x(f(\overline{\phi})) := \forall y [\mathrm{ST}_y(f(\overline{\phi}))\rightarrow xNy]$& $\mathrm{ST}_y(f(\overline{\phi})) := \forall \overline{x^{\epsilon_f}}[\left(\bigwedge_{1\leq i \leq n_f}\mathrm{ST}_{x^{\epsilon_f(i)}}(\phi_i)\right)\rightarrow R_{f}(y,\overline{w^{\epsilon_f}})]$\\
				$\mathrm{ST}_x(g(\overline{\phi})) := \forall \overline{y^{\epsilon_g}}[\left(\bigwedge_{1\leq i \leq n_g}\mathrm{ST}_{y^{\epsilon_g(i)}}(\phi_i)\right)\rightarrow R_{g}(x,\overline{u^{\epsilon_g}})]$ & $\mathrm{ST}_y(g(\overline{\phi})) := \forall x [\mathrm{ST}_x(g(\overline{\phi}))\rightarrow xNy]$\\
			\end{tabular}
	}}
\end{center}
\end{definition}

%
%

The following lemma is proved by a routine induction.

\begin{lemma}\label{lemma:Standard:Translation}
	For any $\mathcal{L}$-model $\mathbb{M}$, any $\mathcal{L}$-frame $\mathbb{F}$,  any $w \in W$, $u \in U$ and for all $\mathcal{L}$-formulas $\phi$ and $\psi$,
	\begin{enumerate}
		
		\item $\mathbb{M}, w \Vdash \phi$ iff $\mathbb{M} \models \mathrm{ST}_x(\phi)[x:= w]$
		
		\item $\mathbb{M}, u \succ \psi$ iff $\mathbb{M} \models \mathrm{ST}_y(\psi)[y:= u]$
		
		\item
		\begin{eqnarray*}
			\mathbb{M} \Vdash \phi \vdash \psi &\mbox{iff} &\mathbb{M} \models \forall x \forall y [(\mathrm{ST}_x(\phi) \wedge  \mathrm{ST}_y(\psi)) \rightarrow xNy]\\
			&\mbox{iff} &\mathbb{M} \models \forall x  [\mathrm{ST}_x(\phi) \rightarrow   \mathrm{ST}_x(\psi) ]\\
			&\mbox{iff} &\mathbb{M} \models \forall y [\mathrm{ST}_y(\psi) \rightarrow   \mathrm{ST}_y(\phi)].
		\end{eqnarray*}

		\item\begin{eqnarray*}
			\mathbb{F} \Vdash \phi \vdash \psi &\mbox{iff} &\mathbb{F} \models \forall \overline{P} \forall x \forall y [(\mathrm{ST}_x(\phi) \wedge  \mathrm{ST}_y(\psi)) \rightarrow xNy]\\
			&\mbox{iff} &\mathbb{F} \models \forall\overline{P}  \forall x  [\mathrm{ST}_x(\phi) \rightarrow   \mathrm{ST}_x(\psi) ]\\
			&\mbox{iff} &\mathbb{F} \models \forall\overline{P}  \forall y [\mathrm{ST}_y(\psi) \rightarrow   \mathrm{ST}_y(\phi)].
		\end{eqnarray*}
		where $\overline{P}$ are  the vectors of all predicate symbols corresponding to propositional variables 
		occurring in $\mathrm{ST}_x(\phi)$, $\mathrm{ST}_y(\phi)$, $\mathrm{ST}_x(\psi)$ and $\mathrm{ST}_y(\psi)$.
		
	\end{enumerate}
	
\end{lemma}

\section{Constructions and morphisms of LE-frames}
\label{sec:constructions}
In the present section we define morphisms, co-products, filter-ideal extensions and ultrapowers of LE-frames. Our approach builds on the category theoretic framework for polarities developed in \cite{moshier2016relational}. We define morphisms as duals of complete homomorphisms of complete LE-algebras. We also define p-morphic images and generated subframes of LE-frames using the dual notions of injective and surjective complete homomorphisms of their associated complex algebras. Throughout this section, we fix an arbitrary LE-signature $\mathcal{L}=\mathcal{L}(\mathcal{F},\mathcal{G})$.
\subsection{Co-products of $\mathcal{L}$-frames}

Let $\{\mathbb{F}_i\mid i\in I\}$ be a family of $\mathcal{L}$-frames, where $\mathbb{F}_i = (W_i, U_i, N_i, \mathcal{R}_\mathcal{F}^{i}, \mathcal{R}_\mathcal{G}^{i})$, and $\mathcal{R}_\mathcal{F}^{i}: = \{R_f^{i}\mid f\in \mathcal{F}\}$ and $\mathcal{R}_\mathcal{G}^{i}: = \{R_g^{i}\mid g\in \mathcal{G}\}$ for each $i\in I$. We let
\[\coprod_{i\in I}\mathbb{F}_i: = (\coprod_{i\in I}W_i, \coprod_{i\in I}U_i, \coprod_{i\in I}N_i, \coprod_{i\in I}\mathcal{R}_\mathcal{F}^{i}, \coprod_{i\in I}\mathcal{R}_\mathcal{G}^{i}), \]
where $\coprod_{i\in I}W_i $ and $\coprod_{i\in I}U_i$ denote the usual disjoint unions of sets (let $\iota_i: W_i\to \coprod_{i\in I}W_i$ and $\gamma_i: U_i\to \coprod_{i\in I}U_i$ denote the canonical injections),
\[\coprod_{i\in I}N_i\subseteq \left(\coprod_{i\in I}W_i\right)\times \left(\coprod_{i\in I}U_i\right)\]

\[\coprod_{i\in I}N_i: = \left(\bigcup_{i\in I}(\iota_i, \gamma_i)[N_i]\right)\cup \left(\bigcup_{i\neq j} (\iota_i, \gamma_j)[W_i\times U_j]\right),\]

where $(\iota_i, \gamma_i)[N_i]: = \{(\iota_i(w), \gamma_i(u))\mid (w, u)\in N_i\}$, and $(\iota_i, \gamma_j)[W_i\times U_j]: = \{(\iota_i(w), \gamma_j(u))\mid (w, u)\in W_i\times U_j\}$. For every $f\in \mathcal{F}$ of arity $n_f = n$,

\[\coprod_{i\in I}R^{i}_f\subseteq \left(\coprod_{i\in I}U_i\right)\times \left(\coprod_{i\in I}W_i\right)^{\epsilon_f}\]
\[\coprod_{i\in I}R^{i}_f: =  \left(\bigcup_{i\in I}(\gamma_i, (\iota_i)^{\epsilon_f})[R_f^{i}]\right)\cup \left(\bigcup_{(i, \overline{j})\in I^{n+1} \text{ not constant}} (\gamma_i, (\iota_{\overline{j}})^{\epsilon_f}) [U_i\times (W_{\overline{j}})^{\epsilon_f}]\right)\]
where $(\gamma_i, (\iota_i)^{\epsilon_f})[R_f^{i}]: = \{(\gamma_i(u), (\iota_i)^{\epsilon_f}(\overline{w^{\epsilon_f}}))\mid (u, \overline{w^{\epsilon_f}})\in R^{i}_f\}$

\begin{lemma}
	For any family $\{\mathbb{F}_i\mid i\in I\}$ of $\mathcal{L}$-frames, 
	$$\left(\coprod_{i\in I}\mathbb{F}_i\right)^+\cong \prod_{i\in I}\left(\mathbb{F}_i\right)^+. $$
\end{lemma}

\begin{example}\label{ex:coproduct}
	Consider the $\mathcal{L}$-frames $\mathbb{F}_i=(\mathbb{P}_i,R_i)$ for $i=1,2$, where $\mathbb{P}_i=(W_i,U_i,N_i)$ and $$W_i=\{a_i,b_i\}\quad U_i=\{x_i,y_i\}\quad N_i=\{(a_i,x_i),(b_i,y_i)\}\quad R_i=\{(a_i,y_i),(b_i,x_i)\}.
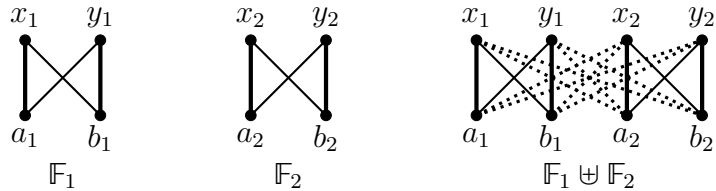
$$
	\begin{center}
		\begin{tikzpicture}
		\filldraw[black] (0,0) circle (2 pt);
		\filldraw[black] (1,0) circle (2 pt);
		\filldraw[black] (0,1) circle (2 pt);
		\filldraw[black] (1,1) circle (2 pt);
		\draw[ultra thick] (0, 0) -- (0, 1);
		\draw[ultra thick] (1, 0) -- (1, 1);
		\draw[thick] (0, 0) -- (1, 1);
		\draw[thick] (1, 0) -- (0, 1);
		\draw (0, -0.3) node {{\large $a_1$}};
		\draw (0, 1.3) node {{\large $x_1$}};
		\draw (1, -0.3) node {{\large $b_1$}};
		\draw (1, 1.3) node {{\large $y_1$}};
		\draw (0.5, -0.8) node {{\large $\mathbb{F}_1$}};
		
		\filldraw[black] (3,0) circle (2 pt);
		\filldraw[black] (4,0) circle (2 pt);
		\filldraw[black] (3,1) circle (2 pt);
		\filldraw[black] (4,1) circle (2 pt);
		\draw[ultra thick] (3, 0) -- (3, 1);
		\draw[ultra thick] (4, 0) -- (4, 1);
		\draw[ thick] (3, 0) -- (4, 1);
		\draw[ thick] (4, 0) -- (3, 1);
		\draw (3, -0.3) node {{\large $a_2$}};
		\draw (3, 1.3) node {{\large $x_2$}};
		\draw (4, -0.3) node {{\large $b_2$}};
		\draw (4, 1.3) node {{\large $y_2$}};
		\draw (3.5, -0.8) node {{\large $\mathbb{F}_2$}};

		\filldraw[black] (6,0) circle (2 pt);
		\filldraw[black] (7,0) circle (2 pt);
		\filldraw[black] (8,0) circle (2 pt);
		\filldraw[black] (9,0) circle (2 pt);
		\filldraw[black] (6,1) circle (2 pt);
		\filldraw[black] (7,1) circle (2 pt);
		\filldraw[black] (8,1) circle (2 pt);
		\filldraw[black] (9,1) circle (2 pt);
		
		\draw[ultra thick] (6, 0) -- (6, 1);
		\draw[ultra thick] (7, 0) -- (7, 1);
		\draw[thick] (6, 0) -- (7, 1);
		\draw[thick] (7, 0) -- (6, 1);
		\draw (6, -0.3) node {{\large $a_1$}};
		\draw (6, 1.3) node {{\large $x_1$}};
		\draw (7, -0.3) node {{\large $b_1$}};
		\draw (7, 1.3) node {{\large $y_1$}};
		
		\draw[ultra thick] (8, 0) -- (8, 1);
		\draw[ultra thick] (9, 0) -- (9, 1);
		\draw[thick] (8, 0) -- (9, 1);
		\draw[ thick] (9, 0) -- (8, 1);
		\draw (8, -0.3) node {{\large $a_2$}};
		\draw (8, 1.3) node {{\large $x_2$}};
		\draw (9, -0.3) node {{\large $b_2$}};
		\draw (9, 1.3) node {{\large $y_2$}};
		\draw[very thick,dotted] (6, 0) -- (8, 1);
		\draw[very thick,dotted] (6, 0) -- (9, 1);
		\draw[very thick,dotted] (7, 0) -- (8, 1);
		\draw[very thick,dotted] (7, 0) -- (9, 1);
		\draw[very thick,dotted] (6, 1) -- (8, 0);
		\draw[very thick,dotted] (6, 1) -- (9, 0);
		\draw[very thick,dotted] (7, 1) -- (8, 0);
		\draw[very thick,dotted] (7, 1) -- (9, 0);
		\draw (7.5, -0.8) node {{\large $\mathbb{F}_1\uplus \mathbb{F}_2$}};
		\end{tikzpicture}
	\end{center}
\end{example}

\subsection{Morphisms of LE-frames}
The following definition is the counterpart of the notion of p-morphism in classical modal logic. It has been obtained as the dual counterpart of the notion of complete homomorphism of $\mathcal{L}$-algebras with an analogous argument as in \cite{dunn2005canonical}.
\begin{definition}\label{def:lehoms}
	Let $\mathcal{L}$ be an LE-language and $\mathbb{F}_1=(\mathbb{W}_1,\mathcal{R}^1_{\mathcal{F}},\mathcal{R}^1_{\mathcal{G}})$ and $\mathbb{F}_2=(\mathbb{W}_2,\mathcal{R}^2_{\mathcal{F}},\mathcal{R}^2_{\mathcal{G}})$ be $\mathcal{L}$-frames. A \emph{p-morphism} of $\mathcal{L}$-frames is a pair $(S,T):\mathbb{F}_1\to \mathbb{F}_2$ such that:
	\begin{itemize}
		\item[p1.]  $S\subseteq W_1\times U_2$ and $T\subseteq U_1\times W_2$;
		\item[p2.] $S^{(0)}[u]$ and $S^{(1)}[w]$ are Galois stable sets in $\mathbb{W}_1$ and $\mathbb{W}_2$ respectively, for every $u\in U_2$ and $w\in W_1$;
		\item[p3.] $T^{(0)}[w]$ and $T^{(1)}[u]$ are Galois stable sets in $\mathbb{W}_1$ and $\mathbb{W}_2$ respectively, for every $u\in U_1$ and $w\in W_2$;
		\item[p4.] $(T^{(0)}[w])^{\downarrow}\subseteq S^{(0)}[w^{\uparrow}]$ for every $w\in W_2$;
		\item[p5.] $T^{(0)}[(S^{(1)}[w])^\downarrow]\subseteq w^\uparrow$ for every $w\in W_1$;
		\item[p6.] $T^{(0)}[((R^2_f)^{(0)}[\overline{w}])^{\downarrow}]=(R^1_f)^{(0)}[\overline{((T^{\epsilon_f})^{(0)}[w])^{\partial}}]$ for every $R^i_f\in\mathcal{R}^i_{\mathcal{F}}$, where $T^{1}=T$ and $T^{\partial}=S$;
		\item[p7.] $S^{(0)}[((R^2_g)^{(0)}[\overline{u}])^{\uparrow}]=(R^1_g)^{(0)}[\overline{((S^{\epsilon_g})^{(0)}[u])^{\partial}}]$ for every $R^i_g\in\mathcal{R}^i_{\mathcal{G}}$, where $S^{1}=S$ and $S^{\partial}=T$.
	\end{itemize}
\end{definition}

\begin{lemma}\label{lem:defofequalmaps}
	For every p-morphism $(S,T):\mathbb{F}_1\to \mathbb{F}_2$ and $a\in( \mathbb{F}_2)^+$  $$(T^{(0)}[\exts{a}])^{\downarrow}=S^{(0)}[\ints{a}].$$
\end{lemma}
\begin{proof}
	One direction follows immediately from p4 of Definition \ref{def:lehoms} and the fact that $\ints{a}=(\exts{a})^\uparrow$. Conversely let us show that $S^{(0)}[\ints{a}]\subseteq (T^{(0)}[\exts{a}])^{\downarrow}$. Since, by p3 and p4, $T^{(0)}[\exts{a}]$ and $S^{(0)}[\ints{a}]$ are stable, this is equivalent to showing that $T^{(0)}[\exts{a}]\subseteq S^{(0)}[\ints{a}]^\uparrow$. Since by p4 $$T^{(0)}[(S^{(1)}[S^{(0)}[\ints{a}]])^\downarrow]\subseteq S^{(0)}[\ints{a}]^\uparrow$$ it is enough to show that $$T^{(0)}[\exts{a}]\subseteq T^{(0)}[(S^{(1)}[S^{(0)}[\ints{a}]])^\downarrow].$$ We have:
	\begin{align*}
	&\quad \ints{a}\subseteq S^{(1)}[S^{(0)}[\ints{a}]]\tag{ Lemma \ref{lem:monochrome}}\\
	\Longrightarrow & \quad  (S^{(1)}[S^{(0)}[\ints{a}]])^\downarrow\subseteq \ints{a}^\downarrow\tag{ antitonicity of $\downarrow$}\\
	\iff & \quad  (S^{(1)}[S^{(0)}[\ints{a}]])^\downarrow\subseteq\exts{a}\tag{$\ints{a}^\downarrow=\exts{a}$}\\
	\Longrightarrow &\quad T^{(0)}[\exts{a}]\subseteq T^{(0)}[(S^{(1)}[S^{(0)}[\ints{a}]])^\downarrow]\tag{antitonicity of $T^{(0)}$}.
\end{align*}
\end{proof}
\begin{definition}
	\begin{enumerate}
		\item Let $(S,T):\mathbb{F}_1\to\mathbb{F}_2$ be a p-morphism. Then we let $$h_{(S,T)}:(\mathbb{F}_2)^+\to(\mathbb{F}_1)^+$$ be defined as $h(a):=(S^{(0)}[\ints{a}],T^{(0)}[\exts{a}])$.
		\item 	Let $h:(\mathbb{F}_1)^+\to(\mathbb{F}_2)^+$ be a complete $\mathcal{L}$-homomorphism. Then we let$$(S_h,T_h):\mathbb{F}_2\to\mathbb{F}_1$$ be defined as $$S_h(w,u)\iff w\in\exts{h(u^{\downarrow\uparrow})} \qquad T_h(u,w)\iff u\in\ints{h(w^{\uparrow\downarrow})}.$$
	\end{enumerate}

\end{definition}

\begin{prop}\label{lemma:lehoms}
	For any $\mathcal{L}$-frames  $\mathbb{F}_1$ and $\mathbb{F}_2$,
	\begin{enumerate}
		\item $h_{(S,T)}$ is a complete $\mathcal{L}$-homomorphism for every p-morphism $(S,T):\mathbb{F}_1\to\mathbb{F}_2$.
		\item $(S_h,T_h)$ is a p-morphism for every complete  $\mathcal{L}$-homomorphism $h:(\mathbb{F}_1)^+\to(\mathbb{F}_2)^+$.
	\end{enumerate}
\end{prop}
\begin{proof}
	Conditions p2, p3 and via Lemma \ref{lem:defofequalmaps} $T^{(0)}$ conditions p4 and p5 guarantee that $h_{(S,T)}$  is well defined and preserves joins and meets. Conditions p6 and p7 immediately imply that $h_{(S,T)}$ preserves $\mathcal{F}$ connectives and $\mathcal{G}$ connectives respectively.
\end{proof}

\begin{prop}\label{lemma:back and forth}
	For all $\mathcal{L}$-frames  $\mathbb{F}_1$ and $\mathbb{F}_2$,
	\begin{enumerate}
		\item $(S, T) = (S_{h_{(S,T)}}, T_{h_{(S,T)}})$  for every p-morphism $(S,T):\mathbb{F}_1\to\mathbb{F}_2$.
		\item $h = h_{(S_h,T_h)}$ for every complete  $\mathcal{L}$-homomorphism $h:(\mathbb{F}_1)^+\to(\mathbb{F}_2)^+$.
	\end{enumerate}
\end{prop}

\begin{definition}
	For every p-morphism $(S,T):\mathbb{F}_1\to\mathbb{F}_2$,
	\begin{enumerate}
		\item $(S,T)$ is \emph{surjective}, in symbols $(S,T):\mathbb{F}_1\twoheadrightarrow\mathbb{F}_2$, if $a\neq b$ implies $S^{(0)}[\ints{a}]\neq S^{(0)}[\ints{b}]$ (or equivalently $T^{(0)}[\exts{a}]\neq T^{(0)}[\exts{b}]$), for every $a,b\in(\mathbb{F}_2)^+$.  In this case we say that $\mathbb{F}_2$ is a \emph{p-morphic image} of $\mathbb{F}_1$.  
		\item $(S,T)$ is \emph{injective}, in symbols $(S,T):\mathbb{F}_1\hookrightarrow\mathbb{F}_2$, if for every $a\in(\mathbb{F}_1)^+$ there exists $b\in(\mathbb{F}_2)^+ $ such that $S^{(0)}[\ints{b}]=\exts{a}$ (or equivalently $T^{(0)}[\exts{b}]=\ints{a}$). In this case we say that $\mathbb{F}_1$ is a \emph{generated subframe} of $\mathbb{F}_2$.
	\end{enumerate}
\end{definition}

In the following examples we consider the LE-signature $\mathcal{L}=\mathcal{L}(\mathcal{F},\mathcal{G})$ where $\mathcal{F}=\varnothing$ and $\mathcal{G}=\{\Box\}$ with $n_\Box=1$ and $\epsilon_\Box(1)=1$.

\begin{example}\label{ex:morphism1}
	Consider the $\mathcal{L}$-frames $\mathbb{F}_i=(\mathbb{P}_i,R_i)$  for $i=1,2$, where $\mathbb{P}_i=(W_i,U_i,N_i)$ and $$W_1=\{a_1,b_1\}\quad U_1=\{x_1,y_1\}\quad N_1=R_1=\{(a_1,x_1),(b_1,y_1)\}$$ and  $$W_2=\{a_2\}\quad U_2=\{x_2,y_2\}\quad N_2=R_2=\{(a_2,x_2)\}.
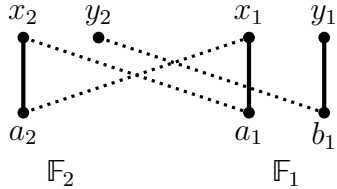
$$
	\begin{center}
		\begin{tikzpicture}
		\filldraw[black] (0,0) circle (2 pt);
		\filldraw[black] (0,1) circle (2 pt);
		\filldraw[black] (1,1) circle (2 pt);
		\draw[ultra thick] (0, 0) -- (0, 1);
		\draw (0, -0.3) node {{\large $a_2$}};
		\draw (0, 1.3) node {{\large $x_2$}};
		\draw (1, 1.3) node {{\large $y_2$}};
		\draw (0.5, -0.8) node {{\large $\mathbb{F}_2$}};
		
		\filldraw[black] (3,0) circle (2 pt);
		\filldraw[black] (4,0) circle (2 pt);
		\filldraw[black] (3,1) circle (2 pt);
		\filldraw[black] (4,1) circle (2 pt);
		\draw[ultra thick] (3, 0) -- (3, 1);
		\draw[ultra thick] (4, 0) -- (4, 1);
		\draw (3, -0.3) node {{\large $a_1$}};
		\draw (3, 1.3) node {{\large $x_1$}};
		\draw (4, -0.3) node {{\large $b_1$}};
		\draw (4, 1.3) node {{\large $y_1$}};
		\draw[very thick,dotted] (0, 0) -- (3, 1);
		\draw[very thick,dotted] (0, 1) -- (3, 0);
		\draw[very thick,dotted] (1, 1) -- (4, 0);
		\draw (3.5, -0.8) node {{\large $\mathbb{F}_1$}};
		\end{tikzpicture}
	\end{center}
	
	Let $(S,T):\mathbb{F}_2\to\mathbb{F}_1$ be the injective p-morphism defined as $$S=\{(a_2,x_1)\}\quad\quad T=\{(x_2,a_1),(y_2,b_1)\}.$$ To see that indeed $(S,T)$ verifies e.g.\ p4 of Definition \ref{def:lehoms}, $(T^{(0)}[a_1])^{\downarrow}=\{a_2\}=S^{(0)}[a_1^{\uparrow}]$ and $(T^{(0)}[b_1])^{\downarrow}=\varnothing=S^{(0)}[b_1^{\uparrow}]$. To see that it is injective, $\{a_2\}=S^{(0)}[x_1]$ and $\varnothing=S^{(0)}[y_1]$. Therefore $\mathbb{F}_2$ is a generated subframe of $\mathbb{F}_1$. 	
\end{example}

\begin{example}\label{ex:morphism2}
	Consider the $\mathcal{L}$-frames $\mathbb{F}_i=(\mathbb{P}_i,R_i)$  for $i=1,2$, where $\mathbb{P}_i=(W_i,U_i,N_i)$ and $$W_1=\{a_1,b_1\}\quad U_1=\{x_1,y_1\}\quad N_1=\{(a_1,x_1),(b_1,y_1)\}\quad R_1=\{(a_1,y_1),(b_1,x_1)\}$$ and  $$W_2=\{a_2\}\quad U_2=\{x_2\}\quad N_2=R_2=\varnothing.$$
	\begin{center}
		\begin{tikzpicture}
		\filldraw[black] (0,0) circle (2 pt);
		\filldraw[black] (1,0) circle (2 pt);
		\filldraw[black] (0,1) circle (2 pt);
		\filldraw[black] (1,1) circle (2 pt);
		\draw[ultra thick] (0, 0) -- (0, 1);
		\draw[ultra thick] (1, 0) -- (1, 1);
		\draw[thick] (0, 0) -- (1, 1);
		\draw[thick] (1, 0) -- (0, 1);
		\draw (0, -0.3) node {{\large $a_1$}};
		\draw (0, 1.3) node {{\large $x_1$}};
		\draw (1, -0.3) node {{\large $b_1$}};
		\draw (1, 1.3) node {{\large $y_1$}};
		\draw (0.5, -0.8) node {{\large $\mathbb{F}_1$}};
		
		\filldraw[black] (3,0) circle (2 pt);
		\filldraw[black] (3,1) circle (2 pt);
		\draw (3, -0.3) node {{\large $a_2$}};
		\draw (3, 1.3) node {{\large $x_2$}};
		\draw (3, -0.8) node {{\large $\mathbb{F}_2$}};
		\end{tikzpicture}
	\end{center}
	
	It can be verified that  $(S,T):\mathbb{F}_1\to\mathbb{F}_2$  defined as $$S=T=\varnothing$$ is a surjective p-morphism. Therefore $\mathbb{F}_2$ is a p-morphic image of $\mathbb{F}_1$. 	
\end{example}

\begin{example}\label{ex:nonmorphism}
		For the same $\mathcal{L}$-frames as Example \ref{ex:morphism2} the pair of relations $(S,T)$ defined as $$S=U_1\times W_2\qquad T=W_1\times U_2$$ is \textbf{not} a p-morphism.
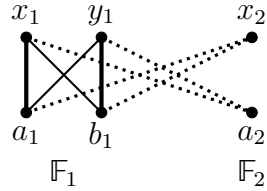
\begin{center}
	\begin{tikzpicture}
	\filldraw[black] (0,0) circle (2 pt);
	\filldraw[black] (1,0) circle (2 pt);
	\filldraw[black] (0,1) circle (2 pt);
	\filldraw[black] (1,1) circle (2 pt);
	\draw[ultra thick] (0, 0) -- (0, 1);
	\draw[ultra thick] (1, 0) -- (1, 1);
	\draw[thick] (0, 0) -- (1, 1);
	\draw[thick] (1, 0) -- (0, 1);
	\draw (0, -0.3) node {{\large $a_1$}};
	\draw (0, 1.3) node {{\large $x_1$}};
	\draw (1, -0.3) node {{\large $b_1$}};
	\draw (1, 1.3) node {{\large $y_1$}};
	\draw (0.5, -0.8) node {{\large $\mathbb{F}_1$}};
	
	\filldraw[black] (3,0) circle (2 pt);
	\filldraw[black] (3,1) circle (2 pt);
	\draw (3, -0.3) node {{\large $a_2$}};
	\draw (3, 1.3) node {{\large $x_2$}};
	\draw[very thick,dotted] (0, 0) -- (3, 1);
	\draw[very thick,dotted] (1, 0) -- (3, 1);
	\draw[very thick,dotted] (0, 1) -- (3, 0);
	\draw[very thick,dotted] (1, 1) -- (3, 0);
	\draw (3, -0.8) node {{\large $\mathbb{F}_2$}};
	\end{tikzpicture}
\end{center}
Indeed, $(T^{(0)}[a_2])^\downarrow=\varnothing\neq\{a_1,b_1\}=S^{(0)}[(a_2)^\uparrow]$ violating the conclusion of Lemma \ref{lem:defofequalmaps}.
\end{example}

\subsection{Filter-ideal frame}
The following definition is the constructive counterpart of the ultrafilter frame (cf.\ \cite[Definition 5.40]{blackburn2002modal}).
\begin{definition}\label{def:f-i-f}
	The \emph{filter-ideal frame} of an $\mathcal{L}$-algebra $\mathbb{A}$ is $\mathbb{F}^\star_{\mathbb{A}}=(\mathfrak{F}_\mathbb{A},\mathfrak{I}_\mathbb{A},N^\star,\mathcal{R}^\star_{\mathcal{F}},\mathcal{R}^\star_{\mathcal{G}})$ defined as follows:
	\begin{enumerate}
		\item $\mathfrak{F}_\mathbb{A}=\{F\subseteq\mathbb{A}\mid F\text{ is a filter}\}$;
		\item $\mathfrak{I}_\mathbb{A}=\{I\subseteq\mathbb{A}\mid I\text{ is an ideal}\}$;
		\item $FN^\star I$ if and only if $F\cap I\neq\varnothing$;
		\item for any $f\in\mathcal{F}$ and any $\overline{F}\in\overline{\mathfrak{F}}^{\epsilon_f}$, $R^\star_f(I,\overline{F})$ if and only $f(\overline{a})\in I$ for some $\overline{a}\in\overline{F}$;
		\item for any $g\in\mathcal{G}$  and any $\overline{I}\in\overline{\mathfrak{I}}^{\epsilon_g}$, $R^\star_g(F,\overline{I})$ if and only if $g(\overline{a})\in F$ for some $\overline{a}\in\overline{I}$. 	
\end{enumerate}
\end{definition}
In order for the definition above to yield an $\mathcal{L}$-frame, we need to verify that the relations $R^\star_f$ and $R^\star_g$ satisfy \eqref{eq:compa1} and \eqref{eq:compa2}. The next lemma verifies this. To simplify the computations we let, for every $\overline{F}\in\overline{\mathfrak{F}}^{\epsilon_f}$ and $\overline{I}\in\overline{\mathfrak{I}}^{\epsilon_g}$, \begin{equation}\label{eq:notation1}g(\overline{I}):=\{g(\overline{a})\in\mathbb{A}\mid \overline{a}\in\overline{I}\}\quad  g^{(i)}(F,\overline{I}^i):=\{b\in\mathbb{A}\mid g(\overline{a}^i_b)\in F\text{ for some }\overline{a}^i\in\overline{I}^i\}\end{equation}
\begin{equation}\label{eq:notation2} f(\overline{F}):=\{f(\overline{a})\in\mathbb{A}\mid \overline{a}\in\overline{F}\}\quad  f^{(i)}(I,\overline{F}^i):=\{b\in\mathbb{A}\mid f(\overline{a}^i_b)\in I\text{ for some }\overline{a}^i\in\overline{F}^i\}\end{equation} 
 
Thanks to this notation, for any $f\in\mathcal{F}$ and $g\in\mathcal{G}$, we can write:
\begin{equation}\label{eq:simplifynotation1}
(R^\star_f)^{(0)}[\overline{F}]=\{I\in\mathfrak{I}\mid f(\overline{F})\cap I\neq\varnothing\}\quad (R^\star_f)^{(i)}[I,\overline{F}^i]=\{H\in\mathfrak{F}^{\epsilon_f(i)}\mid f^{(i)}(I,\overline{F}^i)\cap H\neq\varnothing\} 
\end{equation}
\begin{equation}\label{eq:simplifynotation2}
(R^\star_g)^{(0)}[\overline{I}]=\{F\in\mathfrak{F}\mid g(\overline{I})\cap F\neq\varnothing\}\quad (R^\star)_g^{(i)}[F,\overline{I}^i]=\{H\in\mathfrak{I}^{\epsilon_g(i)}\mid g^{(i)}(F,\overline{I}^i)\cap H\neq\varnothing\}
\end{equation}
For any LE-algebra $\mathbb{A}$ and any $X\subseteq\mathbb{A}$, let $\lfloor X\rfloor$ and $\lceil X\rceil$ respectively denote the filter and ideal generated by $X$. In case $X=\{a\}$ we write $\lfloor a\rfloor$ and $\lceil a\rceil$ for  principal filters and ideals.
\begin{lem}
	\label{lemma:prelimfi}
	For $\mathbb{F}_{\mathbb{A}}$ as above, and any $F\in \mathfrak{F}$, $I\in \mathfrak{I}$:
	\begin{enumerate}
		\item $((R^\star_f)^{(0)}[\overline{F}])^{\downarrow}=\{G\in \mathfrak{F}_\mathbb{A}\mid f(\overline{F})\subseteq G\}$;
		\item $((R^\star_g)^{(0)}[\overline{I}])^{\uparrow}=\{J\in \mathfrak{I}_\mathbb{A}\mid g(\overline{I})\subseteq J\}$;
		\item If $\epsilon_f(i)=1$ then $((R^\star_f)^{(i)}[I_0,\overline{F}^{i}])^{\uparrow}=\{J\in \mathfrak{I}_\mathbb{A}\mid f^{(i)}(I_0,\overline{F}^{i})\subseteq J\}$;
		\item If $\epsilon_f(i)=\partial$ then $((R^\star_f)^{(i)}[I_0,\overline{F}^{i}])^{\downarrow}=\{G\in \mathfrak{F}_\mathbb{A}\mid f^{(i)}(I_0,\overline{F}^{i})\subseteq G\}$;
		\item If $\epsilon_g(i)=1$ then $((R^\star_g)^{(i)}[F_0,\overline{I}^{i}])^{\downarrow}=\{G\in \mathfrak{F}_\mathbb{A}\mid g^{(i)}(F_0,\overline{I}^{i})\subseteq G\}$;
		\item If $\epsilon_g(i)=1\partial$ then $((R^\star_g)^{(i)}[F_0,\overline{I}^{i}])^{\uparrow}=\{J\in \mathfrak{I}_\mathbb{A}\mid g^{(i)}(F_0,\overline{I}^{i})\subseteq J\}$;
		\end{enumerate}
\end{lem}
\begin{proof}
1. Clearly $\{G\in \mathfrak{F}_\mathbb{A}\mid f(\overline{F})\subseteq G\}\subseteq ((R^\star_f)^{(0)}[\overline{F}])^{\downarrow}$. For the converse, assume that $F\notin \{G\in \mathfrak{F}_\mathbb{A}\mid f(\overline{F})\subseteq G\}$, i.e.\ there is some $a\in f(\overline{F})$ such that $a\notin F$. By \eqref{eq:simplifynotation1}, $\lceil a\rceil\in(R^\star_f)^{(0)}[\overline{F}]$, and $\lceil a\rceil\cap F=\varnothing$, therefore $F\notin  ((R^\star_f)^{(0)}[\overline{F}])^{\downarrow}$. The remaining statements are proved analogously.
\end{proof}

\begin{lemma}\label{lem:filtersandideals}
	\begin{enumerate}
		\item If $\epsilon_g(i)=1$ then $g^{(i)}(F,\overline{I}^i)$ is a filter;
		\item If $\epsilon_g(i)=\partial$ then $g^{(i)}(F,\overline{I}^i)$ is an ideal;
		\item If $\epsilon_f(i)=1$ then $f^{(i)}(I,\overline{F}^i)$ is an ideal;
		\item If $\epsilon_f(i)=\partial$ then $f^{(i)}(I,\overline{F}^i)$ is a filter.
	\end{enumerate}
\end{lemma}
\begin{proof}
	1. Assume that $c,d\in g^{(i)}(F,\overline{I}^i)$. That is, there exist $\overline{a}^{i},\overline{b}^{i}\in\overline{I}^i$ such that $g(\overline{a}^{i}_{c})\in F$ and $g(\overline{b}^{i}_{d})\in F$. Since $g$ is meet preserving and join reversing and $F$ is a filter, $g(\overline{(a\land^{\epsilon_g}b)}^{i}_{c\land d})\in F$. Since $\overline{a\land^{\epsilon_g}b}^{i}\in\overline{I}^i$ it follows that $c\land d\in g^{(i)}(F,\overline{I}^i)$. Now assume that $b\in g^{(i)}(F,\overline{I}^i)$, i.e.\ there exists $\overline{a}^{i}\in\overline{I}^i$ such that $g(\overline{a}^{i}_{b})\in F$, and let $b\leq c$. Since $g$ is monotone in the $i$-th coordinate and $F$ is a filter, then $g(\overline{a}^{i}_{c})\in F$. Since $\overline{a}^{i}\in\overline{I}^i$, it follows that $c\in g^{(i)}(F,\overline{I}^i)$.
	The proof of the remaining items are order dual. 
\end{proof}

\begin{prop}
If $\mathbb{A}$ is an LE-algebra, then $\mathbb{F}^\star_{\mathbb{A}}$ is an LE-frame.
\end{prop}
\begin{proof}
	To show that $(R_g^\star)^{(0)}[\overline{I}]$ is stable, i.e.\ $((R_g^\star)^{(0)}[\overline{I}])^{\uparrow\downarrow}\subseteq (R_g^\star)^{(0)}[\overline{I}]$. We have:
\begin{align*}
			((R^\star_g)^{(0)}[\overline{I}])^{\uparrow\downarrow} =&\{J\in \mathfrak{I}_\mathbb{A}\mid g(\overline{I})\subseteq J\}^\downarrow\tag{ Lemma \ref{lemma:prelimfi}.2}\\
			=&\{F\in\mathfrak{F}_{\mathbb{A}}\mid \lceil g(\overline{I})\rceil\cap F\neq\varnothing\}\tag{Definition \ref{def:f-i-f}.3}\\
			\subseteq&\{F\in\mathfrak{F}_{\mathbb{A}}\mid  g(\overline{I})\cap F\neq\varnothing\}\tag{$\ast$}\\
			=&(R^\star_g)^{(0)}[\overline{I}]\tag{\ref{eq:notation1}}\\
\end{align*}

	Let us show the inclusion marked with ($\ast$). Let $F\in\mathfrak{F}_{\mathbb{A}}$ s.t.~$\lceil g(\overline{I})\rceil\cap F\neq\varnothing$. To show that $g(\overline{I})\cap F\neq\varnothing$ it is enough to show that for any $a\in\lceil g(\overline{I})\rceil$ there exists some $b\in  g(\overline{I})$ such that $a\leq b$. Indeed, it is enough to show this for $a=\bigvee_{j\leq k}g(\overline{a_j})$, where $\overline{a_j}\in\overline{I}$ for all $j\leq k$. Notice that $\overline{\bigvee^{\epsilon_g}_{j\leq k}a_j}\in\overline{I}$. Hence $g(\overline{\bigvee^{\epsilon_g}_{j\leq k}a_j})\in g(\overline{I})$. By the tonicity of $g$, we have $a=\bigvee_{j\leq k}g(\overline{a_j})\leq g(\overline{\bigvee^{\epsilon_g}_{j\leq k}a_j})$.  
	
	As for showing that $(R^\star_g)^{(i)}[F,\overline{I}^i]$ is stable, assume that $\epsilon_g(i)=1$. By \eqref{eq:simplifynotation2} $$(R^\star)_g^{(i)}[F,\overline{I}^i]=\{J\in\mathfrak{I}\mid g^{(i)}(F,\overline{I}^i)\cap J\neq\varnothing\}.$$ By Lemma \ref{lem:filtersandideals}.1 $g^{(i)}(F,\overline{I}^i)$ is a filter, therefore $(R^\star_g)^{(i)}[F,\overline{I}^i]=g^{(i)}(F,\overline{I}^i)^{\downarrow}$, which shows that $(R^\star_g)^{(i)}[F,\overline{I}^i]$ is stable. The remaining cases are shown similarly.
\end{proof}

\begin{lem}[cf.\ \cite{gehrke2001bounded} Proposition 2.6]
	$(\mathbb{F}^\star_{\mathbb{A}})^+=\mathbb{A}^\delta$.
\end{lem}


\begin{prop}\label{prop:reversing arrows}
	Let $\mathbb{A}$ and $\mathbb{B}$ be $\mathcal{L}$-algebras.
	\begin{enumerate}
		\item If $h:\mathbb{A}\hookrightarrow\mathbb{B}$ then $(S_{h^\delta},T_{h^\delta}):\mathbb{F}^\star_{\mathbb{B}}\twoheadrightarrow \mathbb{F}^\star_{\mathbb{A}}$.
		\item If $h:\mathbb{A}\twoheadrightarrow\mathbb{B}$ then $(S_{h^\delta},T_{h^\delta}):\mathbb{F}^\star_{\mathbb{B}}\hookrightarrow \mathbb{F}^\star_{\mathbb{A}}$.
	\end{enumerate}
\end{prop}
\begin{proof}
 1. Let $h:\mathbb{A}\hookrightarrow\mathbb{B}$ be an injective $\mathcal{L}$-homomorphism. Then $h^\delta:\mathbb{A}^\delta\hookrightarrow\mathbb{B}^\delta$ is a complete injective $\mathcal{L}$-homomorphism (cf.\ \cite[Lemma 4.9]{gehrke2001bounded}) of complete $\mathcal{L}$-algebras. Then  the p-morphism $(S_{h^\delta},T_{h^\delta}):\mathbb{F}^\star_{\mathbb{B}}\twoheadrightarrow \mathbb{F}^\star_{\mathbb{A}}$ is surjective.
 
 2. Let $h:\mathbb{A}\twoheadrightarrow\mathbb{B}$ be a surjective $\mathcal{L}$-homomorphism. Then $h^\delta:\mathbb{A}^\delta\twoheadrightarrow\mathbb{B}^\delta$ is a complete surjective $\mathcal{L}$-homomorphism (cf.\ \cite[Lemma 4.9]{gehrke2001bounded}) of complete $\mathcal{L}$-algebras. Then  the p-morphism $(S_{h^\delta},T_{h^\delta}):\mathbb{F}^\star_{\mathbb{B}}\hookrightarrow \mathbb{F}^\star_{\mathbb{A}}$ is injective.
\end{proof}

\begin{definition}
	Let $\mathbb{F}$ be an $\mathcal{L}$-frame. The \emph{filter-ideal extension} of $\mathbb{F}$ is the $\mathcal{L}$-frame $\mathbb{F}^\star_{\mathbb{F}^+}$.
\end{definition}

\subsection{Ultrapowers of LE-frames}

Let $\mathbb{F}=(W,U,N,\mathcal{R}_\mathcal{F},\mathcal{R}_\mathcal{G})$ be an $\mathcal{L}$-frame. Let $$\mathcal{L}_\mathbb{F}=\{N, (R_f)_{f\in\mathcal{F}},(R_g)_{g\in\mathcal{G}}\}\cup\{P_{\exts{a}}, P_{\ints{a}} \mid a\in\mathbb{F}^+\}$$ be a first-order language with variables of two sorts, which, for convenience, we denote $W$ and $U$. Henceforth we use $x$ to denote variables of sort $W$ and $y$ to denote variables of sort $U$. Each $P_{\exts{a}}$ is a unary $W$-relation and each $P_{\ints{a}}$ is a unary $U$-relation.  The remaining relations have arity and type compatible with the corresponding relations in $\mathbb{F}$. We expand $\mathbb{F}$ to an $\mathcal{L}_\mathbb{F}$-structure with relations $P_{\exts{a}}$ and $P_{\ints{a}}$ such that $P_{\exts{a}}(w)$ if and only if $w\in \exts{a}$ and $P_{\ints{a}}(u)$ if and only if $u\in\ints{a}$.

\begin{definition}[Power of LE-frame]
	Let $\mathbb{F}$ be an $\mathcal{L}$-frame and let  $J$ be a set of indexes. The \emph{$J$-power of  $\mathbb{F}$} is the following $\mathcal{L}_\mathbb{F}$-structure: $$\mathbb{F}^{J}=(W^{J},U^{J},N^{J},(R^{J}_f)_{f\in\mathcal{F}},(R^{J}_g)_{g\in\mathcal{G}},(P^{J}_{\exts{a}},P^{J}_{\ints{a}})_{a\in\mathbb{F}^+})$$ where:
	\begin{enumerate}
		\item $W^{J}$ is the set of functions $s:J\to W$;
		\item $U^{J}$ is the set of functions $t:J\to U$;
		\item $sN^{J}t$ if and only if $s(j)Nt(j)$ for all $j\in J$;
		\item $R_f(t,\overline{s})$ if and only if $R_f(t(j),\overline{s(j)})$ for all $j\in J$;
		\item $R_g(s,\overline{t})$ if and only if $R_g(s(j),\overline{t(j)})$ for all $j\in J$;
		\item $P^{J}_{\exts{a}}(s)$ if and only if $P_{\exts{a}}(s(j))$ for all $j\in J$;
		\item $P^{J}_{\ints{a}}(y)$ if and only if $P_{\ints{a}}(y(j))$ for all $j\in J$.
	\end{enumerate}
\end{definition}
 For every ultrafilter $\mathcal{U}$ over $J$, let $\equiv_W$ and $\equiv_U$ be the equivalence relations on $W^{J}$ and $U^{J}$ respectively defined as follows:
$$s_1\equiv_W s_2\iff \{j\in J\mid s_1(j)=s_2(j)\}\in\mathcal{U}$$
$$t_1\equiv_U t_2\iff \{j\in J\mid t_1(j)=t_2(j)\}\in\mathcal{U}.$$
We let $[s]$ and $[t]$ respectively denote the $\equiv_W$-equivalence class containing $s$ and the $\equiv_U$-equivalence class containing $t$. We let $W^{\mathcal{U}}$ and $U^{\mathcal{U}}$ denote the resulting quotient sets. It is easy to see that the equivalence relations $\equiv_W$ and $\equiv_U$ are congruences with respect to $N^{J},(R^{J}_f)_{f\in\mathcal{F}},(R^{J}_g)_{g\in\mathcal{G}}$ and $(P^{J}_{\exts{a}},P^{J}_{\ints{a}})_{a\in\mathbb{F}^+}$.
\begin{definition}
 For every $\mathbb{F}$, $J$ and $\mathcal{U}$ as above, the \emph{ultrapower} $$\mathbb{F}^{J}_{/\mathcal{U}}=(W^{\mathcal{U}},U^{\mathcal{U}},N^{\mathcal{U}},(R^{\mathcal{U}}_f)_{f\in\mathcal{F}},(R^{\mathcal{U}}_g)_{g\in\mathcal{G}},(P^{\mathcal{U}}_{\exts{a}},P^{\mathcal{U}}_{\ints{a}})_{a\in\mathbb{F}^+})$$ is the $\mathcal{L}_\mathbb{F}$-structure where:
\begin{enumerate}
	\item $[s]N^{\mathcal{U}}[t]$ if and only if $\{j\in J\mid s(j)Nt(j)\}\in\mathcal{U}$;
	\item $R^{\mathcal{U}}_f([t],\overline{[s]})$ if and only if $\{j\in J\mid R_f(t(j),\overline{s(j)})\}\in\mathcal{U}$;
	\item $R^{\mathcal{U}}_g([s],\overline{[t]})$ if and only if $\{j\in J\mid R_g(s(j),\overline{t(j)})\}\in\mathcal{U}$.
	\item $P^{\mathcal{U}}_{\exts{a}}(s)$ if and only if $\{j\in J\mid P_{\exts{a}}(s(j))\}\in\mathcal{U}$;
	\item $P^{\mathcal{U}}_{\ints{a}}(t)$ if and only if $\{j\in J\mid P_{\ints{a}}(t(j))\}\in\mathcal{U}$.
\end{enumerate}
\end{definition}

Henceforth, we will abuse notation and identify $s$ with $[s]$ and $t$ with $[t]$. We will always use $s$ and $t^\partial$ to denote elements of $W^{\mathcal{U}}$ and $t$ and $s^\partial$ to denote elements of $U^{\mathcal{U}}$.

\begin{thm}[\L os] $$\mathbb{F}^{J}_{/\mathcal{U}}\models_{\mathcal{L}_\mathbb{F}}\varphi(\overline{s},\overline{t})\quad\iff\quad\{j\in J\mid \mathbb{F}\models_{\mathcal{L}_\mathbb{F}}\varphi(\overline{s(j)},\overline{t(j)})\}\in\mathcal{U}.$$
\end{thm}

As an immediate consequence of \L os' Theorem we obtain the following:

\begin{cor}
	For every $\mathbb{F}$, $J$ and $\mathcal{U}$ as above the ultrapower $\mathbb{F}^{J}_{/\mathcal{U}}$ is an $\mathcal{L}$-frame.
\end{cor}
\begin{proof}
	The compatibility conditions can be expressed as $\mathcal{L}_\mathbb{F}$-sentences.
\end{proof}
The following definition is an equivalent reformulation of \cite[beginning of Chapter 5.1]{chang1990model}, cf.\ \cite[Chapter 10.1 Exercise 17]{hodges1993model}.
\begin{definition}
Let $\kappa$ be an infinite cardinal, $M$ be a model  of $\mathcal{L}_\mathbb{F}$, and $\mathcal{L}_\mathbb{F}(M)$ be the language obtained expanding $\mathcal{L}_\mathbb{F}$ with constants symbols for the elements of $M$. Then $M$ is \emph{$\kappa$-saturated} if for any set  $\Sigma$ of formulas in $\mathcal{L}_\mathbb{F}(M)$ such that $\Sigma$ contains 
finitely many free variables $\overline{x}$ and $\overline{y}$ and $|\Sigma|<\kappa$, if $\Sigma$ is finitely satisfied in $M$ then $\Sigma$ is satisfied in $M$.
\end{definition}
\begin{lemma}[cf.\ \cite{chang1990model} Theorem 6.1.8]\label{lem:changultrafilter}
	For any $\mathcal{L}$-frame $\mathbb{F}$ there exists a set $J$ and an ultrafilter $\mathcal{U}$ over $J$ such that $\mathbb{F}^{J}_{/\mathcal{U}}$ is $|\mathcal{L}_\mathbb{F}|^+$-saturated.
\end{lemma}

\section{Enlargement property for LE-logics}
\label{sec:enlargement}
In the classical modal logic setting, the main step of the Goldblatt-Thomason theorem consists in showing that the ultrafilter extension of the disjoint union of a certain family of elements of the class $\mathsf{K}$ of Kripke frames belongs to $\mathsf{K}$. This is done by showing that this ultrafilter extension is the p-morphic image of some ultrapower (cf.\ \cite[Theorem 3.17]{blackburn2002modal}). Goldblatt refers to this existence property as the \emph{enlargement property},  and proves it in the context of polarities (cf.\ \cite[Theorem 6.2]{Go17}).\footnote{In fact, Goldblatt states and proves that there exists an embedding $e:(\mathbb{P}^+)^{\delta}\hookrightarrow(\mathbb{P}^{J}/\mathcal{U})^+$ for some set $J$ and some ultrafilter $\mathcal{U}$ over $J$. The proof for the Kripke frame analogue of this result follows from a construction involving a p-morphism defined on an ultrapower of a structure, and constructs the required embedding as the dual of that p-morphism. This is the strategy we follow in the present paper. However Goldblatt's proof of \cite[Theorem 6.2]{Go17}  does not take this approach, but instead uses \cite[Theorem 3.2]{gehrke2006macneille} about embedding a canonical extension into a MacNeille completion.} In this section we prove the enlargement property for $\mathcal{L}$-frames. In what follows, we fix an LE-signature $\mathcal{L}=\mathcal{L}(\mathcal{F},\mathcal{G})$ and an $\mathcal{L}$-frame $\mathbb{F}$.


\begin{thm}[Enlargement property]\label{thm:embetoup}
	 There exists a surjective p-morphism $(S,T):\mathbb{F}^{J}/\mathcal{U}\twoheadrightarrow\mathbb{F}^\star_{\mathbb{F}^+}$ for some set $J$ and some ultrafilter $\mathcal{U}$ over $J$. 
\end{thm}
\begin{proof}
	The proof will proceed in a series of lemmas, proven below. Let $J$ and $\mathcal{U}$ be as in Lemma \ref{lem:changultrafilter}, i.e.\ such that $\mathbb{F}^{J}_{/\mathcal{U}}$ is $|\mathcal{L}_\mathbb{F}|^+$-saturated. Let $(S,T):\mathbb{F}^J_{/\mathcal{U}}\to F^\star_{\mathbb{F}^+}$ (cf.\ Definition \ref{def:lehoms}), where $S\subseteq W^\mathcal{U}\times\mathfrak{I}_{\mathbb{F}^+}$ and $T\subseteq U^\mathcal{U}\times\mathfrak{F}_{\mathbb{F}^+}$  are defined as follows: \begin{equation}\label{eq:defS}
	sS I\quad\iff\quad s^{-1}[\exts{c}]\in\mathcal{U}\text{ for some }c\in I
	\end{equation}
	\begin{equation}\label{eq:defT}
	tTF\quad\iff\quad t^{-1}[\ints{c}]\in\mathcal{U}\text{ for some }c\in F.
	\end{equation} The relations $S$ and $T$ satisfy the conditions of Definition \ref{def:lehoms}. Indeed, Lemma \ref{lem:compatiblerel} shows that condition p2 and p3 are satisfied. Lemma \ref{lem:twoisone} shows conditions p4 and p5 are satisfied. Lemma \ref{lem:oppres} shows that conditions p6 and p7 are satisfied. Finally, Lemma \ref{lem:injective} implies that  $(S,T)$ is surjective.
\end{proof}

The following two technical lemmas will simplify the further computations.
\begin{lemma}\label{lem:goingdown} The following hold:
	\begin{enumerate}
		\item $(T^{(0)}[F])^{\downarrow}=\{s\in W^\mathcal{U}\mid \{c\in\mathbb{F}^+\mid s^{-1}[\exts{c}]\in\mathcal{U}\}\supseteq F\}$;
		\item $(S^{(0)}[I])^{\uparrow}=\{t\in U^\mathcal{U}\mid \{c\in\mathbb{F}^+\mid t^{-1}[\ints{c}]\in\mathcal{U}\}\supseteq I\}$.
	\end{enumerate}
\end{lemma}
\begin{proof}
	We only prove item 1, the proof of item 2 being dual.  Let $s$ be such that $\{c\in\mathbb{F}^+\mid s^{-1}[\exts{c}]\in\mathcal{U}\}\supseteq F$. Now for every $t\in T^{(0)}[F]$, there exists a $c_t\in F$ such that $t^{-1}[\ints{c_t}]\in\mathcal{U}$. By the definition of $s$ we have that $s^{-1}[\exts{c_t}]\in\mathcal{U}$. Since $\mathcal{U}$ is an ultrafilter  $t^{-1}[\ints{c_t}]\cap s^{-1}[\exts{c_t}]\in\mathcal{U}$. Recall that $wNu$ for every $w\in \exts{c_t}$ and every $u\in\ints{c_t}$. Therefore $s(j)Nt(j)$ for every $j\in  t^{-1}[\ints{c_t}]\cap s^{-1}[\exts{c_t}]$, i.e.\ $sN^\mathcal{U}t$, which shows that $s\in (T^{(0)}[F])^\downarrow$, as required.
	
	For the converse direction, assume contrapositively that $s^{-1}[\exts{c_0}]\notin \mathcal{U}$ for some $c_0\in F$. Since $\mathcal{U}$ is an ultrafilter, this implies that $s^{-1}[W\setminus\exts{c_0}]\in \mathcal{U}$. For every $w\in W\setminus \exts{c_0}$ there exists $u_w\in \ints{c_0}$ such that $(w,u_w)\notin N$. Let $t$ be such that $t(j)=u_{s(j)}$ for $j\in s^{-1}[W\setminus \exts{c_0}]$. Since $t^{-1}[\ints{c_0}]\supseteq  s^{-1}[W\setminus \exts{c_0}]$ it follows that $t^{-1}[\ints{c_0}]\in\mathcal{U}$, i.e. $t\in T^{(0)}[F]$. However, $(s(j),t(j))\notin N $ for every $j\in s^{-1}[W\setminus \exts{c_0}]$, i.e.\ $s\notin (T^{(0)}[F])^{\downarrow}$.\end{proof}

\begin{lemma}\label{lem:fintoone}
	Let $c_1,\ldots,c_n\in\mathbb{F}^+$. For any $w\in W, u\in U, s\in W^\mathcal{U}$ and $t\in U^\mathcal{U}$. The following implications hold:
	\begin{enumerate}
		\item \begin{enumerate}
			\item $\mathbb{F}\models_{\mathcal{L}_\mathbb{F}}P_{\exts{c_1\land\ldots\land c_n}}(w)\quad\iff\quad\mathbb{F}\models_{\mathcal{L}_\mathbb{F}}P_{\exts{c_1}}(w)\land\ldots\land P_{\exts{c_n}}(w)$;
			\item $\mathbb{F}\models_{\mathcal{L}_\mathbb{F}}\lnot P_{\ints{c_1\land\ldots\land c_n}}(u)\quad\Longrightarrow\quad\mathbb{F}\models_{\mathcal{L}_\mathbb{F}}\lnot P_{\ints{c_1}}(u)\land\ldots\land \lnot P_{\ints{c_n}}(u)$;
			\item $\mathbb{F}\models_{\mathcal{L}_\mathbb{F}}P_{\ints{c_1\lor\ldots\lor c_n}}(u)\quad\iff\quad\mathbb{F}\models_{\mathcal{L}_\mathbb{F}}P_{\ints{c_1}}(u)\land\ldots\land P_{\ints{c_n}}(u)$;
			\item $\mathbb{F}\models_{\mathcal{L}_\mathbb{F}}\lnot P_{\exts{c_1\lor\ldots\lor c_n}}(w)\quad\Longrightarrow\quad\mathbb{F}\models_{\mathcal{L}_\mathbb{F}}\lnot P_{\exts{c_1}}(w)\land\ldots\land \lnot P_{\exts{c_n}}(w)$.
		\end{enumerate}
		\item \begin{enumerate}
			\item $\mathbb{F}^{J}/\mathcal{U}\models_{\mathcal{L}_\mathbb{F}}P_{\exts{c_1\land\ldots\land c_n}}(s)\quad\iff\quad\mathbb{F}^{J}/\mathcal{U}\models_{\mathcal{L}_\mathbb{F}}P_{\exts{c_1}}(s)\land\ldots\land P_{\exts{c_n}}(s)$;
			\item $\mathbb{F}^{J}/\mathcal{U}\models_{\mathcal{L}_\mathbb{F}}\lnot P_{\ints{c_1\land\ldots\land c_n}}(t)\quad\Longrightarrow\quad\mathbb{F}^{J}/\mathcal{U}\models_{\mathcal{L}_\mathbb{F}}\lnot P_{\ints{c_1}}(t)\land\ldots\land \lnot P_{\ints{c_n}}(t)$;
			\item $\mathbb{F}^{J}/\mathcal{U}\models_{\mathcal{L}_\mathbb{F}}P_{\ints{c_1\lor\ldots\lor c_n}}(t)\quad\iff\quad\mathbb{F}^{J}/\mathcal{U}\models_{\mathcal{L}_\mathbb{F}}P_{\ints{c_1}}(t)\land\ldots\land P_{\ints{c_n}}(t)$;
			\item $\mathbb{F}^{J}/\mathcal{U}\models_{\mathcal{L}_\mathbb{F}}\lnot P_{\exts{c_1\lor\ldots\lor c_n}}(s)\quad\Longrightarrow\quad\mathbb{F}^{J}/\mathcal{U}\models_{\mathcal{L}_\mathbb{F}}\lnot P_{\exts{c_1}}(s)\land\ldots\land \lnot P_{\exts{c_n}}(s)$.
		\end{enumerate}
	\end{enumerate}
\end{lemma}
\begin{proof}
	\begin{enumerate}
		\item We only show the first two, the remaining two being dual:
		\begin{enumerate}
			\item \begin{align*}
			&	\quad\mathbb{F}\models_{\mathcal{L}_\mathbb{F}}P_{\exts{c_1\land\ldots\land c_n}}(w)\\
			\iff & \quad w\in\exts{c_1\land\ldots\land c_n}\\
			\iff & \quad w\in\exts{c_1}\cap\ldots\cap\exts{c_n}\\
			\iff & \quad \mathbb{F}\models_{\mathcal{L}_\mathbb{F}}P_{\exts{c_1}}(w)\land\ldots\land P_{\exts{c_n}}(w).
			\end{align*}
			\item   \begin{align*}
			&	\quad\mathbb{F}\models_{\mathcal{L}_\mathbb{F}}\lnot P_{\ints{c_1\land\ldots\land c_n}}(u)\\
			\iff & \quad u\notin\ints{c_1\land\ldots\land c_n}\\
			\Longrightarrow & \quad w\notin\ints{c_1}\cup\ldots\cup\ints{c_n}\\
			\iff & \quad \mathbb{F}\models_{\mathcal{L}_\mathbb{F}}\lnot P_{\ints{c_1}}(u)\land\ldots\land \lnot P_{\ints{c_n}}(u).
			\end{align*}
		\end{enumerate}
	\item  	We only show the first two, the remaining two being dual:
	\begin{enumerate}
		\item \begin{align*}
		&	\quad\mathbb{F}^{J}/\mathcal{U}\models_{\mathcal{L}_\mathbb{F}}P_{\exts{c_1\land\ldots\land c_n}}(s)\\
		\iff & \quad  s^{-1}[\exts{c_1\land\ldots\land c_n}]\in\mathcal{U}\\
		\iff & \quad s^{-1}[\exts{c_1}\cap\ldots\cap\exts{c_n}]\in\mathcal{U}\\
		\iff & \quad s^{-1}[\exts{c_1}]\cap\ldots\cap s^{-1}[\exts{c_n}]\in\mathcal{U}\\
		\iff & \quad s^{-1}[\exts{c_1}]\in\mathcal{U}\text{ and }\ldots \text{ and }s^{-1}[\exts{c_n}]\in\mathcal{U}\\
		\iff & \quad\mathbb{F}^{J}/\mathcal{U}\models_{\mathcal{L}_\mathbb{F}}P_{\exts{c_1}}(s)\land\ldots\land P_{\exts{c_n}}(s).
		\end{align*}
		\item   \begin{align*}
		&	\quad\mathbb{F}^{J}/\mathcal{U}\models_{\mathcal{L}_\mathbb{F}}\lnot P_{\ints{c_1\land\ldots\land c_n}}(t)\\
		\iff & \quad  t^{-1}[\ints{c_1\land\ldots\land c_n}]\notin\mathcal{U}\\
		\iff & \quad t^{-1}[(\ints{c_1}\cup\ldots\cup\ints{c_n})^{\downarrow\uparrow}]\notin\mathcal{U}\\
		\Longrightarrow &\quad s^{-1}[(\ints{c_1}\cup\ldots\cup\ints{c_n})]\notin\mathcal{U}\\
		\iff & \quad t^{-1}[\ints{c_1}]\cup\ldots\cup t^{-1}[\ints{c_n}]\notin\mathcal{U}\\
		\iff & \quad t^{-1}[\ints{c_1}]\notin\mathcal{U}\text{ and }\ldots \text{ and }t^{-1}[\ints{c_n}]\notin\mathcal{U}\\
		\iff & \quad\mathbb{F}^{J}/\mathcal{U}\models_{\mathcal{L}_\mathbb{F}}\lnot P_{\ints{c_1}}(t)\land\ldots\land \lnot P_{\ints{c_n}}(t).
		\end{align*}
		\end{enumerate}
\end{enumerate} \end{proof}
\begin{lemma}\label{lem:compatiblerel}
	For every $t\in U^\mathcal{U}$, $s\in W^\mathcal{U}$, $F\in\mathfrak{F}_{\mathbb{F}^+}$ and $I\in\mathfrak{I}_{\mathbb{F}^+}$,
	\begin{enumerate}
		\item the sets $S^{(0)}[I]$ and $S^{(1)}[t]$ are Galois stable;
		\item the sets $T^{(0)}[F]$ and $T^{(1)}[s]$ are Galois stable.
		\end{enumerate}
\end{lemma}
\begin{proof}
	Let us first show that $T^{(1)}[t]$ is Galois stable for every $t\in U^\mathcal{U}$.  Let $$I_t:=\{c\in\mathbb{F}^+\mid t^{-1}[\ints{c}]\in\mathcal{U}\}.$$ Since $\mathcal{U}$ is a filter, $I_t$ is an ideal.  By the definition of $I_t$, $t T F$ if and only if $I_t\cap F\neq\varnothing$ for any filter  $F$. This shows that $T^{(1)}[t]=I_t^{\downarrow}$ which is enough to prove that $T^{(1)}[t]$ is Galois stable.
	
	Now let us show that $(T^{(0)}[F])^{\downarrow\uparrow}=T^{(0)}[F]$. It is enough to show that $(T^{(0)}[F])^{\downarrow\uparrow}\subseteq T^{(0)}[F]$, the converse direction being immediate. Let $t$ be such that $t\notin T^{(0)}[F]$, i.e.\ $t^{-1}[U\setminus \ints{c}]\in\mathcal{U}$ for all $c\in F$. Notice that the set of formulas with a free variable $x$ $$\Sigma:=\{\lnot x N t\}\cup\{P_{\exts{c}}(x) \mid c\in F\} $$ is finitely satisfiable in $\mathbb{F}^{J}_{/\mathcal{U}}$. Indeed, since filters are closed under meets, by Lemma \ref{lem:fintoone}, it is enough to show that for any $c\in F$  the set $S:=\{\lnot x N t, P_{\exts{c}}(x)\}$ is satisfiable. We have $t^{-1}[U\setminus\ints{c}]\in\mathcal{U}$. For every $u\in  U\setminus\ints{c}$ there exists $w_u\in\exts{c}$ such that $(w_u,u)\notin N$. Let $s$ be such that $s(j)= w_{t(j)}$ for $j\in  t^{-1}[U\setminus\ints{c}]$. By definition, if $j\in t^{-1}[U\setminus\ints{c}]$, then $(s(j),t(j))\notin N $ and $P_{\exts{c}}(s(j))$, i.e.\ $\lnot s N t$ and $P_{\exts{c}}(s)$ hold in $\mathbb{F}^{J}_{/\mathcal{U}}$ which finishes the proof that $S$ is satisfiable in $\mathbb{F}^{J}_{/\mathcal{U}}$. Since  $\mathbb{F}^{J}_{/\mathcal{U}}$ is $|\mathcal{L}_\mathbb{F}|^+$-saturated, $\Sigma$ is satisfiable in  $\mathbb{F}^{J}_{/\mathcal{U}}$ as well by assigning the variable $x$ to some witness $s\in W^\mathcal{U}$. By the definition of $\Sigma$, we have that $s^{-1}[\exts{c}]\in\mathcal{U}$ for all $c\in F$, while $(s,t)\notin N^\mathcal{U}$. By Lemma \ref{lem:goingdown}, $s\in (T^{(0)}[F])^\downarrow$. Therefore $t\notin (T^{(0)}[F])^{\downarrow\uparrow}$. This concludes the proof that $T$ is  $N^\mathcal{U}$ and $N^\star$ compatible. The proof for $S$ is dual.
\end{proof}

\begin{lemma}\label{lem:twoisone}
	The following inclusions hold:
	\begin{enumerate}
		\item  $(T^{(0)}[F])^{\downarrow}\subseteq S^{(0)}[F^{\uparrow}]$ for every $F\in \mathfrak{F}_{\mathbb{F}^+}$;
		\item  $T^{(0)}[(S^{(1)}[s])^\downarrow]\subseteq s^\uparrow$ for every $s\in W^\mathcal{U}$.
	\end{enumerate}
\end{lemma}
\begin{proof}
	1. Let $s\in (T^{(0)}[F])^{\downarrow}$. By Lemma \ref{lem:goingdown}, $s^{-1}[\exts{c}]\in\mathcal{U}$ for all $c\in F$. Now let $I\in F^{\uparrow}$. By definition, there exists $c_0\in F$ such that $c_0\in I$. Since $s^{-1}[\exts{c_0}]\in\mathcal{U}$ it follows that $s S I$. Therefore $s\in S^{(0)}[F^{\uparrow}]$.
	
	2. Let $F_s:=\{c\in\mathbb{F}^+\mid s^{-1}[\exts{c}]\in\mathcal{U}\}$. Since $\mathcal{U}$ is a filter, it follows that $F_s$ is a filter. By definition, $s S I$ if and only if $F_s\cap I\neq\varnothing$ for any ideal $I$. This implies that $S^{(1)}[s]=F_s^{\uparrow}$. Therefore $(S^{(1)}[s])^\downarrow=F_s^{\uparrow\downarrow}$, and hence $T^{(0)}[(S^{(1)}[s])^\downarrow]=T^{(0)}[F_s^{\uparrow\downarrow}]=T^{(0)}[F_s]$, the last identity holding because of Lemmas \ref{lem:compatiblerel} and \ref{lem:two definitions of compatibility}. Now let $t\in T^{(0)}[F_s]$. There exists some $c\in F_s$ such that $t^{-1}[\ints{c}]\in\mathcal{U}$. By the definition of $F_s$, we have that $s^{-1}[\exts{c}]\in\mathcal{U}$. For every $w\in \exts{c}$ and $u\in \ints{c}$ we have that $w N u$. Therefore $s(j) N t(j)$ for every $j\in t^{-1}[\ints{c}]\cap s^{-1}[\exts{c}]$. Since $t^{-1}[\ints{c}]\cap s^{-1}[\exts{c}]\in\mathcal{U}$ it follows that $t\in s^\uparrow$.
\end{proof}

\begin{lem}\label{lem:oppres}For every $f\in\mathcal{F}$ and $g\in\mathcal{G}$:
	\begin{enumerate}
	\item $T^{(0)}[((R^\star_f)^{(0)}[\overline{F}])^\downarrow]=(R^\mathcal{U}_f)^{(0)}[\overline{((T^{\epsilon(f)})^{(0)}[F])^{\partial}}]$;
	\item $S^{(0)}[((R^\star_g)^{(0)}[\overline{I}])^\uparrow]=(R^\mathcal{U}_g)^{(0)}[\overline{((S^{\epsilon(f)})^{(0)}[I])^{\partial}}]$.
\end{enumerate}
\end{lem}
\begin{proof}
	1. Let $t\notin T^{(0)}[((R^\star_f)^{(0)}[\overline{F}])^\downarrow]$, that is for any $\overline{a_1},\ldots,\overline{a_n}\in\overline{F}$, if $c=f(\overline{a_1})\land\cdots\land f(\overline{a_n})$ then $t^{-1}[\ints{c}]\notin\mathcal{U}$. Notice that since $c'=f(\overline{a_1\land\cdots\land a_n})\leq f(\overline{a_1})\land\cdots\land f(\overline{a_n})$ we have that $t^{-1}[\ints{c'}]\notin\mathcal{U}$ implies $t^{-1}[\ints{c}]\notin\mathcal{U}$. Because $\overline{a_1\land\cdots\land a_n}\in\overline{F}$, we have equivalently that $t\notin T^{(0)}[((R^\star_f)^{(0)}[\overline{F}])^\downarrow]$ if and only if $t^{-1}[R^{(0)}_f[\overline{\exts{c}}^{\epsilon}]]\notin\mathcal{U}$ for any  $\overline{c}\in\overline{F}^\epsilon$. Notice that the following set of formulas with free variables $\overline{x}^{\epsilon}$ $$\Sigma:=\bigcup_{k\leq m}\{P_{\ints{c}^{\epsilon(k)}}(x_k)\mid c\in F^{\epsilon(k)}\}\cup\{\lnot R_f(t,\overline{x}^{\epsilon})\}$$ is finitely satisfiable in $\mathbb{F}^{J}_{/\mathcal{U}}$. Indeed, since $F^{\epsilon(k)}$ is a filter or ideal it is closed under meets or joins respectively and therefore by Lemma \ref{lem:fintoone} it is enough to show that the set  $$S=\{P_{\exts{c_1}^{\epsilon(1)}}(x^{\epsilon(1)}_1),\ldots,P_{\exts{c_m}^{\epsilon(m)}}(x^{\epsilon(m)}_m),\lnot R_f(t,\overline{x}^{\epsilon})\}$$ is satisfiable.  Since $t^{-1}[R^{(0)}_f(\overline{\exts{c}}^\epsilon)]\notin\mathcal{U}$, we have that $t^{-1}[U\setminus R^{(0)}_f(\overline{\exts{c}}^\epsilon)]\in\mathcal{U}$ and for each $u\in U\setminus R^{(0)}_f(\overline{\exts{c}}^\epsilon)$ and $k\leq m$ there exists $w_u^{\epsilon(k)}\in \exts{c}^{\epsilon(k)}$ such that $\lnot R_f(\overline{w_u}^{\epsilon})$. So let $s^{\epsilon(k)}_k$ be such that $s^{\epsilon(k)}_k(j)=w^{\epsilon(k)}_{t(j)}$ for each $j\in t^{-1}[U\setminus R^{(0)}_f(\overline{\exts{c}}^\epsilon)]$. Then $\overline{s}^{\epsilon}$ satisfy the set $S$. Since $\mathbb{F}^{J}_{/\mathcal{U}}$ is $|\mathcal{L}_\mathbb{F}|^+$-saturated we have that $\Sigma$ is satisfied in $\mathbb{F}^{J}_{/\mathcal{U}}$ as well by assigning the variables $\overline{x}^{\epsilon}$ to some witnesses $\overline{s}^{\epsilon}\in \overline{W}^\mathcal{U}$. Clearly $(s^{\epsilon{k}}_k)^{-1}[\exts{c}^{\epsilon{k}}]\in\mathcal{U}$ for all $c\in F^{\epsilon(k)}_k$, i.e. $\overline{s}^{\epsilon}\in \overline{((T^{\epsilon(f)})^{(0)}[F])^{\partial}}$. Since $\lnot R_f(t,\overline{s}^{\epsilon})$, we have that $t\notin (R^\mathcal{U}_f)^{(0)}[\overline{((T^{\epsilon(f)})^{(0)}[F])^{\partial}}]$.
	
	For the converse direction assume that $t\in T^{(0)}[((R^\star_f)^{(0)}[\overline{F}])^\downarrow]$, i.e. is such that for some $\overline{c}\in\overline{F}^{\epsilon}$ $t^{-1}[R^{(0)}_f[\overline{\exts{c}}^{\epsilon}]]\in\mathcal{U}$. Now let $\overline{s}^{\epsilon}\in\overline{((T^{\epsilon(f)})^{(0)}[F])^{\partial}}$. By Lemma \ref{lem:goingdown} we have that $\overline{s^{-1}}[\overline{\exts{c}}^{\epsilon}]\in\mathcal{U}$. Hence for every $j\in \bigcap\overline{s^{-1}}[\overline{\exts{c}}^{\epsilon}]\cap t^{-1}[R^{(0)}_f[\overline{\exts{c}}^{\epsilon}]]$ we have that $R_f(t(j),\overline{s}^{\epsilon}(j))$. Since $\bigcap\overline{s^{-1}}[\overline{\exts{c}}^{\epsilon}]\cap t^{-1}[R^{(0)}_f[\overline{\exts{c}}^{\epsilon}]]\in\mathcal{U}$ we have that $t\in (R^\mathcal{U}_f)^{(0)}[\overline{((T^{\epsilon(f)})^{(0)}[F])^{\partial}}]$. This concludes the proof of item 1. The proof of item 2 is dual.
\end{proof}

\begin{lemma}\label{lem:injective}
	Let $P,Q\subseteq\mathfrak{F}_{\mathbb{F}^+}$ such that $P^{\uparrow\downarrow}=P$ and $Q^{\uparrow\downarrow}=Q$. If $P\nsubseteq Q$ then $T^{(0)}[Q]\nsubseteq T^{(0)}[P]$.
\end{lemma}
\begin{proof}
	Let $F_0\in P\setminus Q$. Then there exists an ideal $I_0\in Q^\uparrow$ such that $F_0\cap I_0=\varnothing$, since otherwise $F_0\in Q^{\uparrow\downarrow}=Q$. Notice that the set of formulas with a free variable $y$ $$\Sigma:=\{P_{\ints{a}}(y)\mid a\in I_0\}\cup\{\lnot P_{\ints{b}}(y)\mid b\in F_0\}$$ is finitely satisfiable in $\mathbb{F}$. Indeed, since filters are closed under meets and ideals are closed under joins, by Lemma \ref{lem:fintoone}, it is enough to show that for any $a\in I_0$ and $b\in F_0$ the set $S=\{P_{\ints{a}}(y),\lnot P_{\ints{b}}(y)\}$ is satisfiable in $\mathbb{F}$ or show that there exists some $u\in \ints{a}\setminus \ints{b}$.  Now since $F_0\cap I_0=\varnothing$ we have that $b\nleq a$, i.e.\ $\ints{b}\nsupseteqq\ints{a}$, so $\ints{a}\setminus \ints{b}\neq\varnothing$ and so  $u\in \ints{a}\setminus \ints{b}$ exists. Since $\mathbb{F}^{J}_{/\mathcal{U}}$ is $|\mathcal{L}_\mathbb{F}|^+$-saturated we have that $\Sigma$ is satisfied in $\mathbb{F}^{J}_{/\mathcal{U}}$ as well by assigning the variables $y$ to some witness $t\in U^\mathcal{U}$. Since $I_0\in Q^\uparrow$, it follows that for all $F\in Q^{\uparrow\downarrow}=Q$ $I_0\cap F\neq\varnothing$. Hence since $t^{-1}[\ints{a}]\in\mathcal{U}$ for every $a\in I_0$, it follows that $t\in T^{(0)}[Q]$. On the other hand, $F_0\in P$ and $t^{-1}[\ints{b}]\notin\mathcal{U}$ for all $b\in F_0$. Therefore $t\notin T^{(0)}[P]$. This concludes the proof.
\end{proof}

\section{The Goldblatt-Thomason theorem for LE-logics}
\label{sec:GT}
The following proposition is an immediate consequence of Proposition \ref{prop:frames to algebras} and Birkoff's Theorem.
\begin{prop}\label{GT:left to right} Let $\mathcal{L}$ be an LE-signature and let $\varphi\vdash\psi$ be an $\mathcal{L}$-sequent. For all $\mathcal{L}$-frames $\mathbb{F},\mathbb{G},\{\mathbb{F}_i\mid i\in I\}$,
	\begin{enumerate}
		\item If $\mathbb{G}$ is a p-morphic image of $\mathbb{F}$, then $\mathbb{F}\models\varphi\vdash\psi$ implies $\mathbb{G}\models\varphi\vdash\psi$.
		\item If $\mathbb{G}$ is a generated subframe of $\mathbb{F}$, then $\mathbb{F}\models\varphi\vdash\psi$ implies $\mathbb{G}\models\varphi\vdash\psi$.
		\item If $\mathbb{F}$ is the disjoint union of $\{\mathbb{F}_i\mid i\in I\}$, then $\mathbb{F}_i\models\varphi\vdash\psi$ for all $i\in I$ implies $\mathbb{F}\models\varphi\vdash\psi$.
		\item  $\mathbb{F}^\star_{\mathbb{F}^+}\models\varphi\vdash\psi$ implies $\mathbb{F}\models\varphi\vdash\psi$.
	\end{enumerate}
\end{prop}

\begin{thm}\label{th:GT}
	Let $\mathcal{L}=\mathcal{L}(\mathcal{F},\mathcal{G})$ be an LE-signature and let $\mathsf{K}$ be a class of $\mathcal{L}$-frames that is closed under taking ultrapowers. Then $\mathsf{K}$ is $\mathcal{L}$-definable if and only if $\mathsf{K}$ is closed under p-morphic images, generated subframes and co-products, and reflects filter-ideal extensions.
\end{thm}
\begin{proof}
The left to right direction is shown in Proposition \ref{GT:left to right}. For the right to left direction, let $\mathsf{K}$ be any class of frames satisfying the closure conditions of the statement. It suffices to show that any frame $\mathbb{F}$ validating the $\mathcal{L}$-theory of $\mathsf{K}$ is itself a member of $\mathsf{K}$. 

Let $\mathbb{F}$ be such a frame. Clearly $\mathbb{F}^+$ satisfies the theory of $\mathsf{K}^+:=\{\mathbb{G}^+\mid \mathbb{G}\in\mathsf{K}\}$. Hence by
Birkhoff's theorem $\mathbb{F}^+$ belongs to the variety generated by $\mathsf{K}^+$, and therefore $\mathbb{F}^+$ is the homomorphic image of a subalgebra of some product $\prod_{i\in I}\mathbb{F}_i^+\cong(\coprod_{i\in I}\mathbb{F}_i)^+$, where $\mathbb{F}_i\in\mathsf{K}$ for each $i\in I$, as illustrated by the following diagram:
\[\mathbb{F}^+\twoheadleftarrow\mathbb{A}\hookrightarrow(\coprod_{i\in I}\mathbb{F}_i)^+.\] Since $\mathsf{K}$ is closed under taking disjoint unions, $\coprod_{i\in I}\mathbb{F}_i\in\mathsf{K}$. Applying Proposition \ref{prop:reversing arrows} to the diagram above yields:
\[\mathbb{F}^\star_{\mathbb{F}^+}\hookrightarrow\mathbb{F}^\star_{\mathbb{A}}\twoheadleftarrow\mathbb{F}^\star_{(\coprod_{i\in I}\mathbb{F}_i)^+}.\]
By Theorem \ref{thm:embetoup} there exists a set $J$ and some ultrafilter $\mathcal{U}$ over $J$ such that a surjective p-morphism $(\coprod_{i\in I}\mathbb{F}_i)^{J}/\mathcal{U}\twoheadrightarrow\mathbb{F}^\star_{(\coprod_{i\in I}\mathbb{F}_i)^+}$ exists. Since $\coprod_{i\in I}\mathbb{F}_i\in\mathsf{K}$ and $\mathsf{K}$ is closed under ultrapowers, $(\coprod_{i\in I}\mathbb{F}_i)^{J}/\mathcal{U}\in\mathsf{K}$. Since $\mathsf{K}$ is closed under p-morphic images, $\mathbb{F}^\star_{(\coprod_{i\in I}\mathbb{F}_i)^+}\in\mathsf{K}$. As $\mathsf{K}$ is closed under p-morphic images and generated subframes, it follows that $\mathbb{F}^\star_{\mathbb{A}}$ and $\mathbb{F}^\star_{\mathbb{F}^+}$ are in $\mathsf{K}$, which implies that $\mathbb{F}\in\mathsf{K}$ since $\mathsf{K}$ reflects filter-ideal extensions. 
\end{proof}

\section{Applications}
\label{sec:applications}
In the present section, we give examples of first-order conditions on $\mathcal{L}$-frames which we show to be not definable in the corresponding $\mathcal{L}$ language.
Let $\mathcal{L}:=(\mathcal{F},\mathcal{G})$ where $\mathcal{F}=\varnothing$ and $\mathcal{G}=\{\Box\}$. Then $\mathcal{L}$-frames are tuples $\mathbb{F}=(\mathbb{P},R)$ where $\mathbb{P}=(W,U,N)$ is a polarity and $R\subseteq W\times U$ is an $N$-compatible relation.

\begin{example}\label{ex:difference}
Let $\mathsf{K}$ be the elementary class of $\mathcal{L}$-frames $\mathbb{F}$ defined by \begin{equation}\label{eq:difference}
R=N^c.
\end{equation}
To see that $\mathsf{K}$ is not $\mathcal{L}$-definable, consider the $\mathcal{L}$-frames of Example \ref{ex:coproduct}.
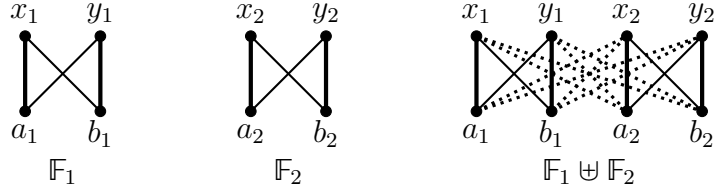
\begin{center}
\begin{tikzpicture}
 \filldraw[black] (0,0) circle (2 pt);
  \filldraw[black] (1,0) circle (2 pt);
   \filldraw[black] (0,1) circle (2 pt);
  \filldraw[black] (1,1) circle (2 pt);
   \draw[ultra thick] (0, 0) -- (0, 1);
   \draw[ultra thick] (1, 0) -- (1, 1);
    \draw[thick] (0, 0) -- (1, 1);
   \draw[thick] (1, 0) -- (0, 1);
    \draw (0, -0.3) node {{\large $a_1$}};
     \draw (0, 1.3) node {{\large $x_1$}};
     \draw (1, -0.3) node {{\large $b_1$}};
     \draw (1, 1.3) node {{\large $y_1$}};
   \draw (0.5, -0.8) node {{\large $\mathbb{F}_1$}};

  \filldraw[black] (3,0) circle (2 pt);
  \filldraw[black] (4,0) circle (2 pt);
   \filldraw[black] (3,1) circle (2 pt);
  \filldraw[black] (4,1) circle (2 pt);
   \draw[ultra thick] (3, 0) -- (3, 1);
   \draw[ultra thick] (4, 0) -- (4, 1);
    \draw[ thick] (3, 0) -- (4, 1);
   \draw[ thick] (4, 0) -- (3, 1);
    \draw (3, -0.3) node {{\large $a_2$}};
     \draw (3, 1.3) node {{\large $x_2$}};
     \draw (4, -0.3) node {{\large $b_2$}};
     \draw (4, 1.3) node {{\large $y_2$}};
   \draw (3.5, -0.8) node {{\large $\mathbb{F}_2$}};

  \filldraw[black] (6,0) circle (2 pt);
  \filldraw[black] (7,0) circle (2 pt);
  \filldraw[black] (8,0) circle (2 pt);
  \filldraw[black] (9,0) circle (2 pt);
  \filldraw[black] (6,1) circle (2 pt);
  \filldraw[black] (7,1) circle (2 pt);
  \filldraw[black] (8,1) circle (2 pt);
  \filldraw[black] (9,1) circle (2 pt);

  \draw[ultra thick] (6, 0) -- (6, 1);
   \draw[ultra thick] (7, 0) -- (7, 1);
    \draw[thick] (6, 0) -- (7, 1);
   \draw[thick] (7, 0) -- (6, 1);
    \draw (6, -0.3) node {{\large $a_1$}};
     \draw (6, 1.3) node {{\large $x_1$}};
     \draw (7, -0.3) node {{\large $b_1$}};
     \draw (7, 1.3) node {{\large $y_1$}};

     \draw[ultra thick] (8, 0) -- (8, 1);
   \draw[ultra thick] (9, 0) -- (9, 1);
    \draw[thick] (8, 0) -- (9, 1);
   \draw[ thick] (9, 0) -- (8, 1);
    \draw (8, -0.3) node {{\large $a_2$}};
     \draw (8, 1.3) node {{\large $x_2$}};
     \draw (9, -0.3) node {{\large $b_2$}};
     \draw (9, 1.3) node {{\large $y_2$}};
   \draw[very thick,dotted] (6, 0) -- (8, 1);
    \draw[very thick,dotted] (6, 0) -- (9, 1);
    \draw[very thick,dotted] (7, 0) -- (8, 1);
    \draw[very thick,dotted] (7, 0) -- (9, 1);
    \draw[very thick,dotted] (6, 1) -- (8, 0);
    \draw[very thick,dotted] (6, 1) -- (9, 0);
    \draw[very thick,dotted] (7, 1) -- (8, 0);
    \draw[very thick,dotted] (7, 1) -- (9, 0);
   \draw (7.5, -0.8) node {{\large $\mathbb{F}_1\uplus \mathbb{F}_2$}};
\end{tikzpicture}
\end{center}

Then, clearly, $(W_1\times U_2)\cup (W_2\times U_1)\subseteq (N_1\uplus N_2)\cap(R_1\uplus R_2)$, which implies that $$R_1\uplus R_2\neq (N_1\uplus N_2)^c.$$ This shows that  $\mathsf{K}$ is not closed under disjoint unions, hence by Theorem \ref{th:GT}, condition \eqref{eq:difference} is not $\mathcal{L}$-definable. 	 
\end{example}

\begin{example}\label{ex:every point has a predecessor}
Let $\mathsf{K}$ be the elementary class of $\mathcal{L}$-frames $\mathbb{F}$ defined by \begin{equation}\label{eq:every point has a predecessor}
\forall u\exists w(\lnot wRu).
\end{equation}
To see that $\mathsf{K}$ is not $\mathcal{L}$-definable consider the $\mathcal{L}$-frames and the p-morphism of Example \ref{ex:morphism1}.
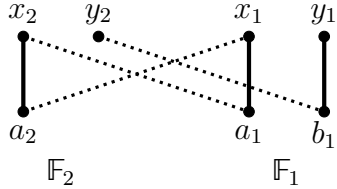
\begin{center}
\begin{tikzpicture}
 \filldraw[black] (0,0) circle (2 pt);
   \filldraw[black] (0,1) circle (2 pt);
  \filldraw[black] (1,1) circle (2 pt);
   \draw[ultra thick] (0, 0) -- (0, 1);
    \draw (0, -0.3) node {{\large $a_2$}};
     \draw (0, 1.3) node {{\large $x_2$}};
     \draw (1, 1.3) node {{\large $y_2$}};
   \draw (0.5, -0.8) node {{\large $\mathbb{F}_2$}};

  \filldraw[black] (3,0) circle (2 pt);
  \filldraw[black] (4,0) circle (2 pt);
   \filldraw[black] (3,1) circle (2 pt);
  \filldraw[black] (4,1) circle (2 pt);
   \draw[ultra thick] (3, 0) -- (3, 1);
   \draw[ultra thick] (4, 0) -- (4, 1);
    \draw (3, -0.3) node {{\large $a_1$}};
     \draw (3, 1.3) node {{\large $x_1$}};
     \draw (4, -0.3) node {{\large $b_1$}};
     \draw (4, 1.3) node {{\large $y_1$}};
      \draw[very thick,dotted] (0, 0) -- (3, 1);
      \draw[very thick,dotted] (0, 1) -- (3, 0);
      \draw[very thick,dotted] (1, 1) -- (4, 0);
   \draw (3.5, -0.8) node {{\large $\mathbb{F}_1$}};
   \end{tikzpicture}
   \end{center}

Since $\mathbb{F}_2$ is a generated subframe of $\mathbb{F}_1$ and $\mathbb{F}_1\in\mathsf{K}$ while $\mathbb{F}_2\notin\mathsf{K}$, the class $\mathsf{K}$ is not closed under generated subframes, hence by Theorem \ref{th:GT}, condition \eqref{eq:every point has a predecessor} is not $\mathcal{L}$-definable. 	
\end{example}

\begin{example}\label{ex:irreflexive}
	Let $\mathsf{K}$ be the elementary class of $\mathcal{L}$-frames $\mathbb{F}$ defined by \begin{equation}\label{eq:irreflexive}
	R^c\subseteq N.
	\end{equation}
	To see that $\mathsf{K}$ is not $\mathcal{L}$-definable consider the $\mathcal{L}$-frames and p-morphism of Example \ref{ex:morphism2}.
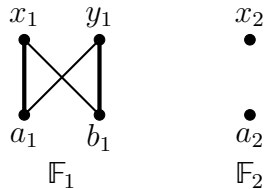
\begin{center}
\begin{tikzpicture}
 \filldraw[black] (0,0) circle (2 pt);
  \filldraw[black] (1,0) circle (2 pt);
   \filldraw[black] (0,1) circle (2 pt);
  \filldraw[black] (1,1) circle (2 pt);
   \draw[ultra thick] (0, 0) -- (0, 1);
   \draw[ultra thick] (1, 0) -- (1, 1);
    \draw[thick] (0, 0) -- (1, 1);
   \draw[thick] (1, 0) -- (0, 1);
    \draw (0, -0.3) node {{\large $a_1$}};
     \draw (0, 1.3) node {{\large $x_1$}};
     \draw (1, -0.3) node {{\large $b_1$}};
     \draw (1, 1.3) node {{\large $y_1$}};
   \draw (0.5, -0.8) node {{\large $\mathbb{F}_1$}};

  \filldraw[black] (3,0) circle (2 pt);
   \filldraw[black] (3,1) circle (2 pt);
    \draw (3, -0.3) node {{\large $a_2$}};
     \draw (3, 1.3) node {{\large $x_2$}};
   \draw (3, -0.8) node {{\large $\mathbb{F}_2$}};
   \end{tikzpicture}
   \end{center}

Since $\mathbb{F}_2$ is a p-morphic image of $\mathbb{F}_1$ and $\mathbb{F}_1\in\mathsf{K}$ while $\mathbb{F}_2\notin\mathsf{K}$,  $\mathsf{K}$ is not closed under  p-morphic images, hence by Theorem \ref{th:GT}, condition \eqref{eq:irreflexive} is not $\mathcal{L}$-definable. 	
\end{example}

\section{Conclusions}
\label{sec:conclusions}
\paragraph{Present contributions. } In the present paper, we state and prove a version of the Goldbatt-Thomason theorem which applies uniformly to normal LE-logics in arbitrary signatures. This class of logics includes well known logics such as the full Lambek calculus and its axiomatic extensions,  orthologic, and the Lambek-Grishin calculus. The theorem is formulated as usual in terms of four model-theoretic constructions (coproduct, bounded morphic image, generated subframe, filter-ideal frame) on LE-frames, which we define and justify on duality-theoretic grounds. 
\paragraph{A wider research program. } In \cite{Go17}, Goldblatt  axiomatically defines a ``canonicity framework'' which is guaranteed to satisfy Goldblatt's algebraic generalisation of Fine's canonicity theorem: an ultraproducts-closed class of structures generates a variety that is closed under canonical extensions.  As a case study, Goldblatt proved that this canonicity framework applies to general lattices.

A natural prosecution of the present work is to apply Goldblatt's canonicity framework  to normal LEs, and more in general to varieties generated by concept lattices with additional operations that are \emph{first-order definable} over polarity-based models. In other words, operations that are definable via a first-order Standard Translation such as the one given in Definition \ref{def:st}. The role of first-order definability is  core to the  relational semantics of wide classes of logics on classical, (bi-)intuitionistic and distributive propositional bases, and  in the setting of LE-logics, the polarity-based semantics is a natural candidate to explore meta-logical properties of LE-logics in connection with first-order definability. 
The results of the present paper can provide a basis where these ideas can be developed.

\paragraph{Labelled calculi for LE-logics. } An example of such meta-logical properties is proof-theoretic and consists in uniformly developing labelled sequent calculi for LE-logics, applying Sara Negri's methodology \cite{negri2005proof, dyckhoff2012proof} in the context of $\mathcal{L}$-frames. 
\bibliography{bib}
\bibliographystyle{plain}

\end{document}